\documentclass{amsart}

\usepackage[T1]{fontenc}
\usepackage{amsmath}
\usepackage{graphicx, amsthm, amssymb, url, gensymb, booktabs, bm}
\usepackage{mathtools}
\usepackage{abraces}% http://ctan.org/pkg/abraces
\usepackage{tikz} 
\usetikzlibrary{matrix}
\usepackage{enumerate} 
\usetikzlibrary{calc}
\usepackage[initials]{amsrefs}
\usepackage{fullpage, setspace}
\usepackage{mathrsfs} % Use \mathscr
\usepackage{anyfontsize} % font size
\usepackage{t1enc} % font size
\usepackage{yfonts} % gothic font style (\textfrak, \textswab, \textgoth)
\usepackage[normalem]{ulem} % \sout
\usepackage{mathtools} % matrix align

\usepackage[colorlinks]{hyperref}
\hypersetup{
linkcolor=blue,      % color of internal links (change box color with linkbordercolor)
citecolor=red      % color of links to bibliography
}

\theoremstyle{plain}
\newtheorem{theorem}{Theorem}[section]

\newtheorem{proposition}[theorem]{Proposition}
\newtheorem{lemma}[theorem]{Lemma}
\newtheorem*{theorem*}{Theorem}
\newtheorem*{lemma*}{Lemma}
\newtheorem*{proposition*}{Proposition}
\newtheorem*{corollary*}{Corollary}

\theoremstyle{definition}
\newtheorem{definition}[theorem]{Definition}

\newtheorem{remark}[theorem]{Remark}

\newtheorem*{definition*}{Definition}
\newtheorem*{example*}{Example}
\newtheorem*{remark*}{Remark}

\newcommand{\TO}{T_{\text{OOCF}}}
\newcommand{\TE}{T_{\text{EICF}}}

\newcommand{\bH}{\mathbb{H}}
\newcommand{\bC}{\mathbb C}
\newcommand{\bR}{\mathbb R}
\newcommand{\bQ}{\mathbb Q}
\newcommand{\bZ}{\mathbb Z}
\newcommand{\bN}{\mathbb N}

\newcommand{\slr}{\text{SL}_2(\mathbb{R})}
\newcommand{\slz}{\text{SL}_2(\mathbb{Z})}

\newcommand{\veps}{\varepsilon}

\DeclareMathOperator{\ldab}{\langle\!\langle}
\DeclareMathOperator{\rdab}{\rangle\!\rangle}

\DeclareMathOperator{\If}{if}

\DeclareMathOperator{\forr}{for}
\DeclareMathOperator{\all}{all}
\DeclareMathOperator{\any}{any}
\DeclareMathOperator{\st}{such~that}
\DeclareMathOperator{\andd}{and}

\DeclareMathOperator{\im}{Im}

\definecolor{brickred}{rgb}{0.8, 0.25, 0.33}

\newcommand{\hocir}[3]{
\filldraw[fill = #3](-5+5*#1/#2,0) arc (-90:270:5*1/#2*1/#2*0.5); \node at (-5+5*#1/#2,-0.2) {$\frac {#1}{#2}$};
}
\newcommand{\hocirr}[3]{
\filldraw[fill = #3](-16+50*#1/#2,0) arc (-90:270:50*1/#2*1/#2*0.5); 
}

\newcommand{\stkout}[1]{\ifmmode\text{\sout{\ensuremath{#1}}}\else\sout{#1}\fi}
\numberwithin{equation}{section}

\begin{document}

\title{Odd-Odd Continued Fraction Algorithm}

\author{Dong Han Kim}
\address{Department of Mathematics Education, Dongguk University--Seoul, 30 Pildong-ro 1-gil, Jung-gu, Seoul, 04620 Korea}
\email{kim2010@dongguk.edu}

\author{Seul Bee Lee}
\address{Centro di Ricerca Matematica Ennio de Giorgi, Scuola Normale Superiore, Piazza dei Cavalieri 3, 56126 Pisa, Italy}
\email{seulbee.lee@sns.it}
%\address{Department of Mathematical Sciences, Seoul National University, Kwanak-ro 1, Kwanak-gu, Seoul, 08826 {  Korea}}
%\email{seulbee.lee@snu.ac.kr}

\author{Lingmin Liao}
\address{Univ Paris Est Creteil, CNRS, LAMA, F-94010 Creteil, France  \& Univ Gustave Eiffel, LAMA, F-77447 Marne-la-Vall\'ee, France}
\email{lingmin.liao@u-pec.fr}

%11J70: Continued fractions and generalizations
%37E05: Maps of the interval

%\thanks{Research partially supported by the National Research Foundation of Korea (NRF-2018R1A2B6001624).}

\begin{abstract}
By using a jump transformation associated to the Romik map, we define a new continued fraction algorithm called odd-odd continued fraction, whose principal convergents are rational numbers of odd denominators and odd numerators. 
Among others, it is proved that all the best approximating rationals of odd denominators and odd numerators of an irrational number are given by the principal convergents of the odd-odd continued fraction algorithm and vice versa. 
\end{abstract} 

\maketitle

\section{Introduction}\label{Sec:Intro}
One main topic of Diophantine approximation studies the approximation of an irrational number by rational numbers.
Given an irrational number $x$, we call a rational $p/q$ \emph{a best approximation of $x$} if
$$\left|qx-p\right| < \left|bx-a\right| ~\forr~\any~ \frac{a}{b}\not=\frac{p}{q} ~\st~ 0<b \le q.$$
%where $p/q$ and $a/b$ are irreducible fractions with positive denominators.
Here, and in the whole paper, by convention, when we write a rational number $p/q$, we always assume that $p \in \bZ$,  $q \in \bN$ and $p$ and $q$ are coprime.
The  celebrated Lagrange Theorem (see \cite[Chapter II]{RoSz92} and \cite[Section 6]{Khi63}) states that the best approximations of an irrational number $x$ are the convergents, i.e., the finite truncations, of the 
\emph{regular continued fraction} (RCF) of $x$:
\begin{equation}\label{Eq:RCF}
x= d_0+\cfrac{1}{d_1+\cfrac{1}{{d_2+\cfrac{1}{\ddots}}}}, %\quad \quad ~\orr~ \quad d_0+\cfrac{1}{d_1+\cfrac{1}{\ddots+\cfrac{1}{d_{n-1}+\cfrac{1}{d_n}}}}% \quad \text{with} \ d_0\in \mathbb{Z},  \ \text{and} \ d_j\in \mathbb{N} \ (j\geq 1).
\end{equation}
where $d_0\in\bZ$ and $d_j\in\bN$, for $j\ge 1$.
%We denote it by 
%$$[d_0;d_1,d_2\cdots,d_j,\cdots].$$
%It is known that every irrational number $x$ has an RCF expansion which means that there exists $\{d_j\}_{j\ge 1}$ such that the following sequence
More precisely, a rational $p/q$ is a best approximation of an irrational $x$ if and only if it is one of the convergents:
\begin{equation*}
\frac{p^R_0}{q^R_0}:=d_0, ~ \frac{p^R_1}{q^R_1}:=d_0+\frac{1}{d_1}, ~\frac{p^R_2}{q^R_2}:=d_0+\frac{1}{d_1+\frac{1}{d_2}}, ~\cdots
\end{equation*}

Let $\bH = \{z\in\bC: \im(z)>0\}$ be the %Poincar\'{e} 
upper half-plane. % model of the hyperbolic plane.
The group $\slr$ acts on $\mathbb{H}$ as isometries defined by
$$g = \begin{pmatrix}a&b\\c&d\end{pmatrix}: z \mapsto g(z) = \frac{az+b}{cz+d}.$$
%which we call \emph{M\"{o}bius transformations}.
%The RCF is related to the action of $\slz$.
There is a close connection between the geodesics on the modular surface $\slz\backslash\bH$ and  the RCF algorithm (e.g.   \cite{series1985modular}).
%The RCF is related to the action of $\slz$ on $\slr$.
Especially, the orbit $\slz(\infty)=\bQ$ corresponds to a unique cusp of $\slz\backslash\bH$. 
Let $$
\Theta = \left\{\begin{pmatrix}a&b\\c&d\end{pmatrix} \in \slz \ : \ \begin{pmatrix}a&b\\c&d\end{pmatrix}\equiv \begin{pmatrix}1&0\\0&1\end{pmatrix} \text{ or }  \begin{pmatrix}0&-1 \\ 1 &0\end{pmatrix} \pmod 2 \right\}.
$$
Then $\Theta$ is a subgroup of $\slz$ of index $3$ and the quotient space $\Theta\backslash\bH$ is a hyperbolic surface with two cusps corresponding to the orbits $\Theta(\infty)$ and $\Theta(1)$ of $\infty$ and $1$. 
Kraaikamp--Lopes \cite{KrLo96} and Boca--Merriman \cite{BoMe18} found that the geodesics on $\Theta\backslash\bH$ is strongly related to the
%The modular surface $\slz\backslash\bH$ has a unique cusp which corresponds to a class of parabolic fixed points $\slz.\infty$. 
\emph{even integer continued fraction} (EICF) introduced by Schweiger \cite{Sch82, Sch84}, which is a continued fraction with even integers such that 
%was investigated by Schweiger in \cite{Sch82}, \cite{Sch84}. 
%An EICF is of the form
\begin{equation}\label{Eq:eicf}
b_0 + \cfrac{\eta_0}{b_1 + \cfrac{\eta_1}{b_2 +  \ddots } }, %= : [(b_1, \eta_1), (b_2, \eta_2), \dots ],
\end{equation}
where $b_0\in 2\bZ$, $b_i \in 2\mathbb N$ for $i\ge 1$ and $\eta_i \in \{ -1, +1\}$.
%{\color{teal}It was introduced by Schweiger in \cite{Sch82}, \cite{Sch84}.}

We classify rational numbers into two classes by the orbits $\Theta(\infty)$ and $\Theta(1)$.
If $p/q\in \Theta(\infty)$, then $p$ and $q$ are of different parity.
If $p/q \in \Theta(1)$, then $p$ and $q$ are both odd. 
We call a rational number in $\Theta(\infty)$ \emph{an $\infty$-rational} and a rational number in $\Theta(1)$ \emph{a $1$-rational}. The proportion of odd/even, even/odd and odd/odd in the RCF convergents was investigated by Moeckel \cite{Moe82}. Further, the asymptotic density of the RCF convergents whose denominators and numerators satisfying congruence equations was obtained by Jager-Liadet \cite{JaLi88}.

%{\color{blue}Moeckel \cite{Moe82} investigated the proportion of odd/even, even/odd and odd/odd in the RCF principal convergents with the principal congruence subgroup $\Gamma(2)$.}
%{\color{teal}$\leftarrow$ Lee: I added the citations according to comment 8 but I am afraid that it does not look natural. Also, I am not sure how to cite Jager-Liadet's result because they treated Dirichlet's theorem. We might give citations as follows:}
%{ DK: I think Jager-Liadet also treated  {\color{teal}$\rightarrow$ Lee: Ok. You mean the following sentence is ok?:}}
%{\color{blue}There are results of the asymptotic density of the RCF principal convergents whose denominators and numerators satisfying congruence equations \cite{Moe82}, \cite{JaLi88}.}
%{\color{teal} But I am not sure it is a good way.}

%There are two different types of rationals, one is called $\infty$-rationals which are of different parity of numerator and denominator, and the other is $1$-rationals which have odd numerator and odd denominator. Denote by $\Theta(\infty)$ and $\Theta(1)$ respectively these two different types of rationals. (See Section Section \ref{Sec:Map} for the reason of the name and the notation of $\infty$-rationals and $1$-rationals).
Short and Walker \cite{ShWa14} defined \emph{a best $\infty$-rational approximation of $x$} by a rational $p/q\in\Theta(\infty)$ satisfying
\begin{equation}\label{Eq:SW}
\left|qx-p\right| < \left|bx-a\right| ~\forr~\any~ \frac{a}{b}\in\Theta(\infty) \text{ apart from } \frac{p}{q} ~\st~ 0<b \le q,
\end{equation}
and showed that the best $\infty$-rational approximations are convergents of EICF.  

Our motivation of the paper is to study the best $1$-rational approximations of an irrational number defined as follows. %Analogous to Short and Walker's best $\infty$-rational approximations, we have the following definitions.
\begin{definition}
For $x\in\bR\setminus\bQ$, $p/q \in \Theta(1)$ is a \emph{best $1$-rational approximation} of $x$ if 
\begin{equation}\label{Eq:bestapp}
|qx-p|< |bx-a| ~\forr~\any \frac{a}{b}\in\Theta(1) \text{ apart from } \frac{p}{q} \st 0<b\le q.
\end{equation}
\end{definition}
We introduce a new continued fraction, called the odd-odd continued fraction (OOCF, see Section \ref{Sec:Map}) of the form 
$$
 1 - \cfrac{1}{a_1 + \cfrac{\veps_1}{ 2 - \cfrac{1}{a_2 + \cfrac{\veps_2}{ 2 - \cfrac{1}{\ddots}}}}},
$$
where $a_n \in \mathbb N$, and $\veps_n \in \{1,-1\}$ for $a_n \ge 2$ and $\veps_n = 1$ for $a_n = 1$.
Our first main theorem is the following.
\begin{theorem}\label{Thm:main1}
A fraction $p/q$ is a best $1$-rational approximation of an irrational number $x$ if and only if it is one of the principal convergents of the odd-odd continued fraction of $x$.
\end{theorem}

For RCF, Lagrange and Euler proved that an irrational number has eventually periodic RCF if and only if it is a quadratic irrational. (See \cite[Chapter III-\S1]{RoSz92} and \cite[Section 10]{Khi63}). For OOCF, we have the following second main theorem.
\begin{theorem}\label{Thm:main2}
An eventually periodic OOCF expansion converges to an $\infty$-rational or a quadratic irrational.
Moreover, a quadratic irrational has an eventually periodic OOCF expansion.
\end{theorem}

We also investigate the relation between the OOCF and the RCF. We show that for any real number $x$, the principal convergents of its OOCF are intermediate convergents of its RCF (Theorem \ref{Thm:intermediate}). Further, we can convert RCF expansions into OOCF expansions (Theorem \ref{Thm:convert}). 

\medskip
Our paper is organized as follows. In Section~\ref{Sec:Map}, we introduce the OOCF algorithm and give some basic properties of OOCFs. %\sout{the OOCF map} $T_\mathrm{OOCF}$.
In Section~\ref{Sec:OOCF} we study the principal convergents of OOCFs, and prove Theorem \ref{Thm:main2}.
Section~\ref{Sec:BestA} is devoted to the proof of Theorem \ref{Thm:main1}. %, we prove that all the best approximating 1-rationals are obtained by the principal convergents of OOCFs, and vice versa. 
The relations between the OOCF expansions and the RCF expansions are described in the last section. %Section~\ref{Sec:RCF}.  

\section{OOCF algorithm}\label{Sec:Map}
It is known that the {\it partial quotients} $d_j=d_j(x)$ of the RCF of an irrational number $x$ as in \eqref{Eq:RCF} can be generated by the \emph{Gauss map} $G: [0,1] \to [0,1]$ defined by 
$$
 G (x) = \left\{ \frac 1x \right\} \ \text{for }x\in(0,1], \quad \text{and} \quad G(0)=0,
$$
where $\{ \cdot \}$ is the fractional part.
In fact, for an irrational $x$ we have $d_j(x) = \lfloor 1/G^{j-1}(x) \rfloor$ for $j\ge 1$ with $\lfloor \cdot \rfloor$ being the integer part. Further, Gauss map is a {jump transformation} associated to the \emph{Farey map} defined by 
$$
F (x) = 
\begin{cases}
\dfrac{x}{1-x} &\If~ 0 \le x \le \dfrac 12, \vspace{0.2cm} \\
\dfrac{1-x}{x} &\If~ \dfrac 12 \le x \le 1. \\
\end{cases}
$$ 
In general, let $U: [0,1]\to [0,1]$ be a map and $E$ be a subset of $[0,1]$.
\emph{The first hitting time of $x\in [0,1]$ to $E$} is defined by 
$$n_E(x):=\min\{j\ge 0: U^j(x)\in E\}.$$
A map $J:[0,1]\to[0,1]$ is called \emph{the jump transformation associated to $U$ with respect to $E$} (e.g. \cite[Chapter 19]{schweiger1995ergodic}) if
$$J(x) = U^{n_E(x)+1}(x) ~\forr~\all~ x\in[0,1].$$%\setminus \bQ.$$
We can easily check that $G$ is the jump transformation associated to $F$ with respect to $E=\{0\}\cup (1/2,1]$. In fact, 
$$G(x) = F^{n_{E}(x)+1}(x) ~\forr~ x\in [0,1].$$
We also note that $n_E(x)+1$ is exactly the first partial quotient $a_1$ of the RCF expansion of $x$.

%An \emph{even integer continued fraction} (EICF) is a continued fraction with even integers such that 
%%was investigated by Schweiger in \cite{Sch82}, \cite{Sch84}. 
%%An EICF is of the form
%\begin{equation}\label{Eq:eicf}
%b_0 + \cfrac{\eta_0}{b_1 + \cfrac{\eta_1}{b_2 +  \ddots } }, %= : [(b_1, \eta_1), (b_2, \eta_2), \dots ],
%\end{equation}
%where $b_0\in 2\bZ$, $b_i \in 2\mathbb N$ for $i\ge 1$ and $\eta_i \in \{ -1, +1\}$. 
Similar to the RCF, the partial quotients of EICF in \eqref{Eq:eicf} can be obtained by the EICF map $T_\mathrm{EICF}:[0,1]\to[0,1]$ defined by
$$
T_\mathrm{EICF} (x) = 
\begin{cases}
\dfrac{1}{x} - 2k &\text{ if }  ~\dfrac{1}{2k+1} \le x \le \dfrac{1}{2k}, \vspace{0.2cm} \\
2k - \dfrac{1}{x} &\text{ if }   ~\dfrac{1}{2k} \le x \le \dfrac{1}{2k-1},
\end{cases}\quad \text{for all }k \in \bN,
\quad \andd \quad T_\mathrm{EICF} (0) =0.
$$
The map $T_\mathrm{EICF}$ turns out to be a jump transformation of the following \emph{Romik map} 
\begin{equation}\label{Eq:Romik}
R(x) = 
\begin{cases}
\dfrac{x}{1-2x} &\If~ 0 \le x \le \dfrac 13, \vspace{0.2cm}\\
\dfrac1{x} - 2 &\If~ \dfrac 13 \le x \le \dfrac 12, \vspace{0.2cm}\\
2 - \dfrac1{x} &\If~ \dfrac 12 \le x \le 1, \\
\end{cases}
\end{equation}
introduced by Romik in \cite{Rom08}. In fact,  letting $E_1:=\{0\}\cup [1/3,1]$, we have 
\begin{equation*}
T_\mathrm{EICF} (x) = R^{n_{E_1}(x)+1} (x) \ ~\forr~ x \in [0,1]. %\stkout{\quad \text{ where }n_E(x) = \min \{ j \ge 0 : R^j (x) \in E \}. {\color{red} x=0?}}
\end{equation*}
The Romik map was used to investigate an algorithm generating the Pythagorean triples by multiplying matrices \cite{Ber34, Alp05, Bar63, Con, CNT}.
%Here, 0 represents rational numbers of $\infty$-rationals and 1 represents 1-rationals.
Some number theoretical properties of the Romik map were recently shown in \cite{CK1, CK2}.
%\les{The map $\TE$ is the jump transformation associated to $R$ with respect to $E_1:=[1/3,1]$, i.e.,}
%The EICF map is the acceleration of $R$ \lee{with respect to the hitting} of $E := [\frac 13, 1] $.
%In fact, we have
Panti \cite{Pan20} studied the connection of the Romik map with the billiards in the hyperbolic plane.
%{\color{teal}Panti considered a continued fraction generated by the Romik map and its periodicity in connection with the periodic trajectories of billiards in the hyperbolic plane \cite{Pan20}.} 

\smallskip
Instead of $E_1$, we choose $E_2 = [0, 1/2]\cup \{1\}$ %\sout{which is corresponding to a cusp $\infty$ of $\Theta\backslash\bH$.}
and define  $T_\mathrm{OOCF}:[0,1]\to[0,1]$, called the \emph{odd-odd continued fraction (OOCF) map}, by the jump transformation associated to the Romik map $R$ with respect to $E_2$, i.e., 
$$ T_\mathrm{OOCF} (x) = R^{n_{E_2}(x)+1} (x) \ \text{ for }x\in[0,1].$$
%We define $T_\mathrm{OOCF} (1)=1.$
By simple calculation, we have
\begin{equation}\label{TOOCF}
\quad T_\mathrm{OOCF} (x) = 
\begin{cases}
\dfrac{k x - (k-1) }{k - (k+1) x} &\If~  \dfrac{k-1}{k} \le x \le \dfrac{2k -1}{2k+1}, \vspace{0.2cm} \\
\dfrac{k - (k+1) x}{k x - (k-1) } &\If~   \dfrac{2k -1}{2k+1} \le x \le \dfrac{k}{k+1},
\end{cases}
\quad\text{for all }k\ge 1
\quad \text{and} \quad
T_\mathrm{OOCF} (1)=1.  
\end{equation}
The graph of $\TO$ is shown in Figure~\ref{Fig_T}.
%Evidently, the map $T_\mathrm{OOCF}$ has \lee{an indifferent fixed point at 0} \les{(see Figure~\ref{Fig_T}).}
%\les{
% \sout{since the principal convergents of the continued fractions induced by $\TO$ have odd denominators and odd numerators.}
%%We will explain more detail later.
%}

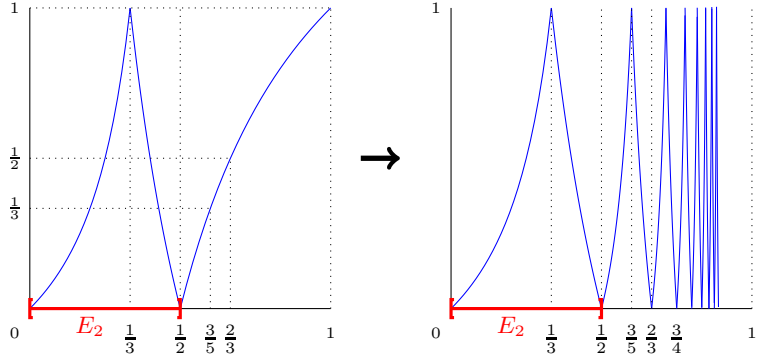
\begin{figure}[t]
\begin{tikzpicture}[scale = 4]
\draw (0,0) -- (1,0);
\draw (0,0) -- (0,1);
\draw[domain=0:1/3, blue] plot (\x, {\x/(1-2*\x)});
\draw[domain=1/3:1/2, blue] plot (\x,{1/\x-2});
\draw[domain=1/2:1, blue] plot (\x, {2-1/\x});

\draw[dotted, thin] (1/2,0) -- (1/2,1);

\node at (-0.05,1){\tiny$1$};
\node at (1,-0.08){\tiny$1$};
\node at (-0.05,-0.08){\tiny$0$};
\node at (-0.05,1/2){\tiny$\frac 12$};
\node at (-0.05,1/3){\tiny$\frac 13$};

\draw[dotted, thin] (1/3,0) -- (1/3,1);

\node at (1/2,-0.1) {\small$\frac12$};
\node at (1/3,-0.1) {\small$\frac13$};
\node at (2/3,-0.1) {\small$\frac23$};
\node at (3/5,-0.1) {\small$\frac35$};

\draw[red, very thick] (0,0) -- (1/2,0);
\draw[red, very thick] (0.01,0.03) -- (0,0.03) -- (0,-0.03) -- (0.01,-0.03);
\draw[red, very thick] (1/2-0.01,0.03) -- (1/2,0.03) -- (1/2,-0.03) -- (1/2-0.01,-0.03);

\node[red] at (0.2,-0.06){\small$E_2$};

\draw[dotted, thin](3/5,0)--(3/5,1/3);
\draw[dotted, thin](2/3,0)--(2/3,1/2);
\draw[dotted, thin](0,1)--(1,1);
\draw[dotted, thin](1,0)--(1,1);
\draw[dotted, thin](0,1/2)--(2/3,1/2);
\draw[dotted, thin](0,1/3)--(3/5,1/3);

\draw[->, line width = 2pt] (1.1,0.5) -- (1.23,0.5);

\draw (1.4,0) -- (1.4+1,0);
\draw (1.4+0,0) -- (1.4+0,1);
\draw[domain=1.4+0:1.4+1/3, blue] plot (\x, {(\x-1.4)/(1-2*(\x-1.4))});
\draw[domain=1.4+1/3:1.4+1/2, blue] plot (\x,{1/(\x-1.4)-2});
%\draw[domain=1.4+1/2:1.4+1, blue] plot (\x, {2-1/(\x-1.4)});
\draw[domain=1.4+1/2:1.4+3/5, blue] plot (\x, {(2*(\x-1.4)-1)/(2-3*(\x-1.4))});
\draw[domain=1.4+3/5:1.4+2/3, blue] plot (\x, {(2-3*(\x-1.4))/(2*(\x-1.4)-1)});
\draw[domain=1.4+2/3:1.4+5/7, blue] plot (\x, {(3*(\x-1.4)-2)/(3-4*(\x-1.4))});
\draw[domain=1.4+5/7:1.4+3/4, blue] plot (\x, {(3-4*(\x-1.4))/(3*(\x-1.4)-2)});
\draw[domain=1.4+3/4:1.4+7/9, blue] plot (\x, {(4*(\x-1.4)-3)/(4-5*(\x-1.4))});
\draw[domain=1.4+7/9:1.4+4/5, blue] plot (\x, {(4-5*(\x-1.4))/(4*(\x-1.4)-3)});
\draw[domain=1.4+4/5:1.4+9/11, blue] plot (\x, {(5*(\x-1.4)-4)/(5-6*(\x-1.4))});
\draw[domain=1.4+9/11:1.4+5/6, blue] plot (\x, {(5-6*(\x-1.4))/(5*(\x-1.4)-4)});
\draw[domain=1.4+5/6:1.4+11/13, blue] plot (\x, {(6*(\x-1.4)-5)/(6-7*(\x-1.4))});
\draw[domain=1.4+11/13:1.4+6/7, blue] plot (\x, {(6-7*(\x-1.4))/(6*(\x-1.4)-5)});
\draw[domain=1.4+6/7:1.4+13/15, blue] plot (\x, {(7*(\x-1.4)-6)/(7-8*(\x-1.4))});
\draw[domain=1.4+13/15:1.4+7/8, blue] plot (\x, {(7-8*(\x-1.4))/(7*(\x-1.4)-6)});
\draw[domain=1.4+7/8:1.4+15/17, blue] plot (\x, {(8*(\x-1.4)-7)/(8-9*(\x-1.4))});
\draw[domain=1.4+15/17:1.4+8/9, blue] plot (\x, {(8-9*(\x-1.4))/(8*(\x-1.4)-7)});

%\draw[domain=1.4:1.4+1/3, blue] plot (\x, {(\x-1.4)/(1-2*(\x-1.4))});
%\draw[domain=1.4+1/3:1.4+1/2, blue] plot (\x,{1/(\x-1.4)-2});
%\draw[domain=1.4+1/2:1.4+1, blue] plot (\x, {2-1/(\x-1.4)});
\draw[dotted, thin] (1.4+1,0)--(1.4+1,1);
\draw[dotted, thin] (1.4+1/2,0) -- (1.4+1/2,1);

\node at (1.4-0.05,1){\tiny$1$};
\node at (1.4+1,-0.08){\tiny$1$};
\node at (1.4-0.05,-0.08){\tiny$0$};
\node at (1.4+1/2,-0.1) {\small$\frac12$};
%\node at (1.4-0.05,1/2){\tiny$\frac 12$};
%\node at (1.4-0.05,1/3){\tiny$\frac 13$};

\draw[dotted, thin] (1.4+1/3,0) -- (1.4+1/3,1);
\draw[dotted, thin] (1.4+3/5,0) -- (1.4+3/5,1);
\draw[dotted, thin] (1.4+2/3,0) -- (1.4+2/3,1);
\node at (1.4+1/3,-0.1) {\small$\frac13$};
\node at (1.4+3/5,-0.1) {\small$\frac35$};
\node at (1.4+2/3,-0.1) {\small$\frac23$};
\node at (1.4+3/4,-0.1) {\small$\frac34$};

\draw[red, very thick] (1.4,0) -- (1.4+1/2,0);
\draw[red, very thick] (1.4+0.01,0.03) -- (1.4+0+0,0.03) -- (1.4+0+0,-0.03) -- (1.4+0+0.01,-0.03);
\draw[red, very thick] (1.4+1/2-0.01,0.03) -- (1.4+1/2,0.03) -- (1.4+1/2,-0.03) -- (1.4+1/2-0.01,-0.03);

\node[red] at (1.4+0.2,-0.06){\small$E_2$};
%\draw[dotted, thin](1.4+1/4,0)--(1.4+1/4,1);
%\draw[dotted, thin](1.4+1/5,0)--t(1.4+1/5,1);
%\draw[step = 0.1, very thin, color=gray!40] (0,0) grid (2.5,1); \node at (0,0){$\bullet$};
\end{tikzpicture}
\caption{The graph of $R$ (left) and the graph of $T_\mathrm{OOCF}$ (right)}\label{Fig_T}
\end{figure}

Using $T_{\text{OOCF}}$, we can induce an OOCF expansion of $x\in[0,1]$.
For convenience, let $T:=\TO$.
%\lee{
We partition $[0,1]$ into the subintervals defined by
\begin{equation}\label{Eq:partition}
B(k+1,-1):=\left[\frac{k-1}{k},\frac{2k-1}{2k+1}\right]\quad \text{ and }\quad B(k,1):=\left[\frac{2k-1}{2k+1},\frac{k}{k+1}\right] \quad \text{ for }k\in\bN.
\end{equation}
%{ \sout{Denote the set of endpoints in \eqref{Eq:partition}, which are $1$-rationals and $\infty$-rationals, respectively, }
%\begin{equation}\label{Eq:singularpts}
%\sout{\mathcal{U}:=\bigcup_{k\geq 1} \left\{{2k-1 \over 2k+1}\right\}\quad \text{ and }\quad \mathcal{V}:=  \bigcup_{k\geq 1} \left\{{k \over k+1}\right\}.} %, \text{ respectively.}
%\end{equation}}}
By \eqref{TOOCF}, if $x\neq 1$, then 
\begin{equation}\label{iteration:1and2}
\dfrac{1}{1-x}=
\begin{cases}
{(k+1)+\dfrac{-1}{2-(1-Tx)}} & \If~x\in B(k+1,-1), \vspace{0.2cm} \\ %\left[\dfrac{k-1}{k},\dfrac{2k-1}{2k+1}\right], \vspace{0.2cm}\\
{k+\dfrac{1}{2-(1-Tx)}} & \If~x\in B(k,1),%\left[\dfrac{2k-1}{2k+1},\dfrac{k}{k+1}\right].
\end{cases}
\quad\text{for all }k\ge 1.
\end{equation}
%{\color{red}{(Lee: This sentence seems to be similar to the next sentence.$\rightarrow$ )}}
Thus, for all $x \notin \bigcup_{i =0}^{n-1} T^{-i}( \{ 1 \} )$
% $x\in [0,1] \setminus \bigcup_{n\geq 0} T^{-n}(\mathcal{U} \cup \mathcal{V}\cup\{0\})$,
\begin{align}\label{eq:iteration}
1-x = \cfrac{1}{a_1 + \cfrac{\veps_1}{ 2 - \cfrac{1}{{{\ddots \cfrac{\ddots}{a_n + \cfrac{\veps_n}{2- (1-T^nx)}}}}}}},
\end{align}
where 
{
\begin{equation}\label{Eq:letter}
(a_n,\veps_n)=
\begin{cases}
(k+1,-1), ~ \text{ if } T^{n-1}(x)\in B(k+1,-1), \vspace{0.2cm} \\%\left[\dfrac{k-1}{k},\dfrac{2k-1}{2k+1}\right],\vspace{0.2cm}\\
(k,1) \text{,}~\quad\quad\quad\text{if } T^{n-1}(x)\in B(k,1). %\left[\dfrac{2k-1}{2k+1},\dfrac{k}{k+1}\right].
\end{cases}
\end{equation}

In particular, if $x=0$, then $T(0)=0$ and we have 
$$1-x=\cfrac{1}{2+\cfrac{-1}{2-(1-Tx)}}.$$
Thus, $0$ has a unique infinite OOCF expansion: $0=[\![(2,-1), (2,-1), \cdots]\!]= [\![(2,-1)^\infty]\!]$.
If $x = \frac{k}{k+1} \in B(k,1) \cap B(k+2,-1)$ for some $k \ge 1$. Then $T(x) = 0$ and we have both choices in \eqref{iteration:1and2}:
$$
1-x= \frac{1}{k+1} = \cfrac{1}{k+2+\cfrac{-1}{2-(1-Tx)}}= \cfrac{1}{k+\cfrac{1}{2-(1-Tx)}}.$$
Hence, a rational number $x={k \over k+1}$ has two infinite OOCF expansions: $x=[\![(k+2, -1),(2,-1)^\infty]\!]$ and  $x=[\![(k, 1),(2,-1)^\infty]\!]$. Furthermore, for any $x\in T^{-n}(\{0\})$ with some $n\geq 2$, we can apply the iteration (\ref{iteration:1and2}) $n-1$ times, and then write $1-T^{n-1}(x)$ in two different ways. Therefore, any $x\in \bigcup_{n\geq 1} T^{-n}(\{0\})$ has two infinite OOCF expansions ending with $(2,-1)^\infty$.

%(PLEASE CHECK THEM!)

Let $x={2k-1 \over 2k+1} \in B(k+1,-1) \cap B(k,1)$ for some $k \ge 1$. Then $T(x) = 1 $ and
$$1-x=\cfrac{1}{(k+1)+\cfrac{-1}{2-(1-Tx)}}=\cfrac{1}{(k+1)+\cfrac{-1}{2}}, \quad \text{or} \quad 1-x=\cfrac{1}{k+\cfrac{1}{2-(1-Tx)}}=\cfrac{1}{k+\cfrac{1}{2}}.$$
Thus $x={2k-1 \over 2k+1}$ has two finite OOCF expansions: $x=[\![(k+1,-1)]\!]$ and $x=[\![(k,1)]\!]$. 
Furthermore, for any $x \in T^{-n} ( \{ 1\})$ with some $n \geq 1$ and $x \ne 1 $, we can apply the iteration (\ref{iteration:1and2}) $n-1$ times, and then write $1-T^{n-1}(x)$ in two different ways. 
Therefore, any $x\in \bigcup_{n\geq 0} T^{-n} (\{ 1 \} ) \setminus \{ 1\}$ has two finite OOCF expansions which differ at the last partial quotient.
 }

Note that for a rational $m/n$, the denominator of $\TO(m/n)$ is strictly less than $n$. 
Note also that $\TO$ sends a $1$-rational to a $1$-rational and an $\infty$-rational to an $\infty$-rational.
Thus if $m/n$ is a $1$-rational, then $\TO^N(m/n)=1$ for some $N\geq 1$, while if $m/n$ is a non-zero $\infty$-rational, then $\TO^N(m/n)=0$ for some $N\geq 1$. 
%
%Hence any $1$-rational belongs to 
%$$\TO^{-N}(1)= \TO^{-N+1} (\mathcal{U}), \quad \text{for some } N\geq 1$$  
%and any $\infty$-rational belongs to $$\TO^{-N}(0)= \TO^{-N+1} (\mathcal{V}) \quad \text{for some } N\geq 1.$$ Finally, we note that 
%\[
%\bigcup_{n= 0}^\infty \TO^{-n}(\mathcal{U} \cup \mathcal{V}\cup\{0\})=\bigcup_{n=1}^\infty\TO^{-n}(\{0,1\})=\bQ.
%\] 
%
Hence any $1$-rational and $\infty$-rational belongs to 
{ $\TO^{-N}(\{1\})$ and $\TO^{-N} (\{ 0 \})$ for some  $N\geq 1$ respectively. 
Finally, we note that 
\[
\bigcup_{n=1}^\infty\TO^{-n}(\{0,1\})=\bQ.
\] }
Then for any $x\in [0,1] \setminus \mathbb Q$, we can iterated \eqref{eq:iteration} infinitely and uniquely to get its OOCF expansion:
$$
x = 1 - \cfrac{1}{a_1 + \cfrac{\veps_1}{ 2 - \cfrac{1}{a_2 + \cfrac{\veps_2}{ 2 - \cfrac{1}{\ddots}}}}}.%{ 2 - \cfrac{1}{a_2 + \cfrac{\veps_2}{ 2 - \cfrac{1}{a_3 + \cfrac{\veps_3}{2- \ddots}}}}}}.
$$
where $a_n \in \mathbb N$ and $\veps_n \in \{1,-1\}$ for $a_n \ge 2$ and $\veps_n = 1$ for $a_n = 1$. 

We denote the OOCF expansion of $x$ by 
$$x=[\![(a_1,\veps_1),(a_2,\veps_2),\cdots,(a_n,\veps_n),\cdots]\!],$$ 
and call $(a_n,\veps_n)$ the \emph{$n$-th partial quotients} of $x$ in its OOCF expansion.
By the above discussion, we have the following proposition. 

\begin{proposition}\label{Thm:1-rat_quad}
The following properties hold.
\begin{enumerate}
%\item\label{Prop:1-rat:finite} \sout{Any finite OOCF is a $1$-rational.  {\color{red} put it later?}}
\item\label{Prop:1-rat:inftyrat} Any non-zero $\infty$-rational has exactly two infinite OOCF expansions ending with $(2,-1)^\infty$.
\item\label{Prop:1-rat_quad:1-rat=>finite} Each $1$-rational has exactly two finite OOCF expansions which differ only in the last partial quotient.
%Any finite OOCF is a $1$-rational. Every $1$-rational has exactly two OOCF expansions, and they are finite and only differ in the last digit. 
\item\label{Prop:1-rat:irr} Every irrational has a unique infinite OOCF expansion.
\end{enumerate}
\end{proposition}

To end this section, we remark that the two maps $\TO$ and $\TE$ are conjugate.
Define $f:[0,1]\to[0,1]$ by $f(x) := \dfrac{1-x}{1+x}.$
%\begin{theorem}\label{Thm:conj}
%More precisely, let $f$ be the function on the interval $[0,1]$ defined by $x\mapsto\frac{1-x}{1+x}$. 
Then the map $\TO$ is conjugate to $\TE$ via $f$, i.e., $$f\circ \TO\circ f^{-1}=\TE.$$
Schweiger \cite{Sch82} proved that $\TE$ admits an ergodic absolutely continuous invariant measure:
%ergodic invariant measure which is absolutely continuous with respect to the Lebesgue measure: 
$d\mu:=\dfrac{dx}{1-x^2}$. 
%Thus, by Theorem \ref{Thm:conj}, we conclude that 
Thus, the measure $f^{-1}_*\mu$ is an ergodic absolutely continuous invariant measure with respect to $T_\mathrm{OOCF}$.  
%By simple calculation, one can find that this measure $f^{-1}_*\mu$ is nothing but the one in Proposition \ref{Prop:invariant}.
%Since $f$ is a measure-preserving map between the two systems $\left((0,1),\frac{dx}{x}\right)$ and $\left((0,1),\frac{dx}{1-x^2}\right)$, $\TO$ is also ergodic.
%\end{remark}
%
%\lee{Note that the absolutely continuous invariant measure for $\TE$ is given by $\frac{dx}{1-x^2}$.}
%\lee{To find a density function of the absolutely continuous invariant measure,} 
Denote by $y = f(x)$. We have %Then $dy = | f'(x)|dx = \frac{2dx}{(1+x)^2}$, thus
$$
\frac{dx}{1-x^2} = \frac{(1+x)^2dy}{2(1-x^2)} = \frac{(1+x)dy}{2(1-x)} = \frac{dy}{2y}.
$$
Hence, we have the following conclusion. 
\begin{proposition}\label{Prop:invariant}
The map $T_\mathrm{OOCF} : [0,1] \to [0,1]$ preserves an infinite ergodic absolutely continuous invariant measure $\dfrac{1}{x}dx.$ 
\end{proposition}
We also remark that the infinity of the absolutely continuous invariant measure comes from the fact that the map $T_\mathrm{OOCF}$ has $0$ as an indifferent fixed point.

\section{Convergents of the odd-odd continued fraction algorithm}\label{Sec:OOCF}
For OOCF, we have three types of convergents by truncating the OOCF in three different places.
%In this section, we consider three forms of truncated continued fractions
%which give us three types of convergents. 
We will investigate the basic properties of such convergents of OOCF.
% and give geometrical interpretations of the hyperbolic plane.
%We need to define principal convergents of OOCF properly 
%We define principal convergents of OOCF as a truncated continued fraction form of an OOCF expansion. 

Let $x\in(0,1)$ be a real number such that
$$x=[\![(a_1,\veps_1),(a_2,\veps_2),\cdots, (a_n, \veps_n),\cdots]\!].$$
%\sout{From now on, when we write a rational number such as $p/q$, we always assume that $p \in \bZ$,  $\don{q} \in \bN$ and \don{$p$} and \don{$q$} are coprime.}
For $n\ge 1$, the \emph{$n$-th principal convergent} of OOCF is defined by
$$
\frac{p_n}{q_n}=[\![(a_1,\veps_1),(a_2,\veps_2),\cdots,(a_n,\veps_n)]\!]=
1 -\cfrac{1}{a_1 + \cfrac{\veps_1}{ 2 - \cfrac{1}{\ddots \cfrac{\veps_{n-1}}{ 2 - \cfrac{1}{ a_n + \cfrac{\veps_n}{2}}}}}}.
$$
We denote
$$
\frac{p'_n}{q'_n}:=1 -\cfrac{1}{a_1 + \cfrac{\veps_1}{ 2 - \cfrac{1}{\ddots \cfrac{\veps_{n-1}}{ 2 - \cfrac{1}{ a_n }}}}} 
\quad\quad
\text{and}
\quad \quad
\frac{p''_n}{q''_n}:=1 -\cfrac{1}{a_1 + \cfrac{\veps_1}{ 2 - \cfrac{1}{\ddots \cfrac{\veps_{n-1}}{ 2 - \cfrac{1}{ a_n + \varepsilon_n }}}}},
$$
and call them the \emph{$n$-th sub-convergent} and \emph{$n$-th pseudo-convergent}, respectively.

To study the convergents of a continued fraction, we have the following general lemma proved by induction (see \cite[p. 3]{Kra91} for details).

\begin{lemma}\label{Lem:generalrecursive}
Consider a general infinite continued fraction and its truncated continued fraction of the form
$$g_0 + \cfrac{e_0}{g_1 + \cfrac{e_1}{ g_2 + \cfrac{e_2}{g_3 + \ddots}}} \quad \quad \and \quad 
\frac{r_n}{s_n} := g_0 + \cfrac{e_0}{g_1 + \cfrac{e_1}{ g_2 + \cfrac{e_2}{\ddots + \cfrac{e_{n-1}}{g_n}}}},$$
where $g_n\in\bZ$ and $|e_n|=1$.
%Consider a general continued fraction of the form
%$$x = g_0 + \cfrac{e_0}{g_1 + \cfrac{e_1}{ g_2 + \cfrac{e_2}{g_3 + \ddots}}},$$
%where $g_n$ is an integer and $|e_n|=1$.
%Let $r_n/s_n$ be the finite continued fraction of the form
%$$ g_0 + \cfrac{e_0}{g_1 + \cfrac{e_1}{ g_2 + \cfrac{e_2}{\ddots + \cfrac{e_{n-1}}{g_n}}}}.$$
Then the following matrix relation holds:
\begin{equation*}\label{eq3}
\begin{pmatrix}r_n & e_n r_{n-1} \\ s_n & e_n s_{n-1} \end{pmatrix}=
\begin{pmatrix}g_0 & e_0 \\ 1 & 0 \end{pmatrix}
\begin{pmatrix}g_1 & e_1 \\ 1 & 0 \end{pmatrix}
\cdots
\begin{pmatrix}g_n & e_n \\ 1 & 0 \end{pmatrix}.
\end{equation*}
Consequently, we have the following recursive formulas:
\begin{equation*}\label{eq2}\left\{\begin{array}{l}
r_{n}=g_n r_{n-1}+e_{n-1} r_{n-2},\\
s_{n}=g_n s_{n-1}+e_{n-1} s_{n-2},
\end{array}\right.\end{equation*}
where $r_{-1}=1$, $s_{-1}=0$, $r_0=g_0$ and $s_0=1$. 
\end{lemma}

Denote the inverse of $\TO|_{B(a_n,\veps_n)}$ by $f_{(a_n,\veps_n)}$. Then by \eqref{iteration:1and2}, we have 
\begin{equation}\label{Eq:inverse}
f_{(a_n,\veps_n)}(t) = 1 - \cfrac{1}{a_n + \cfrac{\veps_n}{1+t}}.
\end{equation}
%The fibered system:
%$$\{f_{(a,\veps)}:a\in\bZ,\veps=\pm1\}$$
%\begin{lemma}\label{Lem:x_between_anti_mid_convergents}
%Let $x=[\![(a_1,\veps_1),(a_2,\veps_2),\cdots]\!]$.
%Then, $x$ is between $p'_n/q'_n$ and $p''_n/q''_n$.
%\end{lemma}
%
%\begin{proof}
%Since $a_n=k+1$ if $\frac{k-1}{k} \le T^{n-1}_\textrm{OOCF}(x) \le \frac{2k -1}{2k+1}$ or $a_n=k$ if $\frac{2k -1}{2k+1} \le T^{n-1}_\textrm{OOCF}(x) \le \frac{k}{k+1}$, 
%we have that $1-T^{n-1}_\textrm{OOCF}(x)$ is between $\frac{1}{a_n}$ and $\frac{1}{a_n'}$.
%\end{proof}
The map $f_{(a_n,\veps_n)}$ corresponds to a linear fractional map on the upper half-plan $\mathbb{H}$ given by the matrix
\begin{equation}\label{Aae1}
A_{(a_n,\varepsilon_n)} :=
\begin{pmatrix}1 & -1 \\ 1 & 0 \end{pmatrix}
\begin{pmatrix}a_n & \veps_n \\ 1 & 0 \end{pmatrix}
\begin{pmatrix}1 & 1 \\ 0 & 1 \end{pmatrix}=
\begin{pmatrix}a_n-1 & a_n+\varepsilon_n-1 \\ a_n & a_n+\varepsilon_n \end{pmatrix}
\in \Theta\cup \begin{pmatrix} 0 & 1 \\ 1 & 0 \end{pmatrix} \Theta.
\end{equation}
%Note that $A_{(a_n,\varepsilon_n)}$ belongs to $\Theta\cup \begin{pmatrix} 0 & 1 \\ 1 & 0 \end{pmatrix} \Theta$. %With \eqref{Aae2} we obtain the following theorem. 
%\les{(Seulbee: We can combine \eqref{eq:recursivematrix} and \eqref{eq:recursivematrix2}. I am not sure it is better to leave them separately, or combine them or remove \eqref{eq:recursivematrix2} as the reviewer's comment.)}
By Lemma~\ref{Lem:generalrecursive}, we have
\begin{multline}\label{eq:recursivematrix}
A_{(a_1,\veps_1)}A_{(a_2,\veps_2)}\cdots A_{(a_n,\veps_n)} \begin{pmatrix} 1 & -1 \\ 1 & 0 \end{pmatrix} \\
=\begin{pmatrix} 1 & -1 \\ 1 & 0 \end{pmatrix}
\begin{pmatrix} a_1 & \veps_1 \\ 1 & 0 \end{pmatrix}
\begin{pmatrix} 2 & -1 \\ 1 & 0 \end{pmatrix}
\begin{pmatrix} a_2 & \veps_2 \\ 1 & 0 \end{pmatrix}
\cdots
\begin{pmatrix} 2 & -1 \\ 1 & 0 \end{pmatrix}
\begin{pmatrix} a_n & \veps_n \\ 1 & 0 \end{pmatrix}
\begin{pmatrix} 2 & -1 \\ 1 & 0 \end{pmatrix}
=
\begin{pmatrix} p_n & -p_n' \\ q_n & -q_n' \end{pmatrix}.
\end{multline}
and
\begin{multline}\label{eq:recursivematrix2}
 A_{(a_1,\veps_1)}A_{(a_2,\veps_2)}\cdots A_{(a_n,\veps_n)}\begin{pmatrix} 0 \\ 1 \end{pmatrix} \\%\begin{pmatrix} 1 & 1 \\ 0 & 1 \end{pmatrix}^{-1}\begin{pmatrix} 1 \\ 1 \end{pmatrix} \\
=\begin{pmatrix} 1 & -1 \\ 1 & 0 \end{pmatrix}
\begin{pmatrix} a_1 & \veps_1 \\ 1 & 0 \end{pmatrix}
\begin{pmatrix} 2 & -1 \\ 1 & 0 \end{pmatrix}
\begin{pmatrix} a_2 & \veps_2 \\ 1 & 0 \end{pmatrix}
\cdots
\begin{pmatrix} 2 & -1 \\ 1 & 0 \end{pmatrix}
\begin{pmatrix} a_n & \veps_n \\ 1 & 0 \end{pmatrix}
\begin{pmatrix} 1  \\ 1 \end{pmatrix}
=
\begin{pmatrix} p_n''  \\ q_n''  \end{pmatrix}.
\end{multline}
These mean that under the linear fractional map of the matrix $A_{(a_1,\veps_1)}A_{(a_2,\veps_2)}\cdots A_{(a_n,\veps_n)}$, the images of $1$, $\infty$ and $0$ are $p_n/q_n$, $p_n'/q_n'$ and $p_n''/q_n''$, respectively.
%Then, the image of $1$ under the matrix as in \eqref{eq:recursivematrix} is $p_n''/q_n''$.

Since $A_{(a_1,\veps_1)}A_{(a_2,\veps_2)}\cdots A_{(a_n,\veps_n)}$ is contained in $\Theta$ or $\begin{pmatrix}0&1\\1&0\end{pmatrix}\Theta$, 
we deduce the following proposition.
\begin{proposition}\label{Thm:parity_of_convergents}
We have
$$\frac{p_n}{q_n} \in \Theta(1) \quad\text{ and }\quad \frac{p'_n}{q'_n}, \ \frac{p''_n}{q''_n} \in \Theta(\infty).$$
\end{proposition}
We remark that the name of odd-odd continued fraction comes from the fact that the principal convergents ${p_n}/{q_n}$ are $1$-rationals, i.e., of odd denominators and odd numerators.  We also remark that by Proposition \ref{Thm:parity_of_convergents}, any finite OOCF is a $1$-rational.

By \eqref{eq:recursivematrix} and \eqref{eq:recursivematrix2}, we have the following recursive relations of the three types of convergents.

\begin{lemma}\label{Lem:recursive}
Let $p_0'=1$, $q_0'=0$, $p_0=1$ and $q_0=1$.
We have the following recursive formulas:
\begin{equation}\label{Eq:p_n'p_n''p_n} \begin{cases}
p_{n}'=a_{n}p_{n-1}-p_{n-1}', \\
q_{n}'=a_{n}q_{n-1}-q_{n-1}',
\end{cases}
\
\begin{cases}
p_{n}''=p_{n}'+\veps_{n}p_{n-1}, \\
q_{n}''=q_{n}'+\veps_{n}q_{n-1}
\end{cases}
\text{ and } \
\begin{cases}
p_{n}=2p_{n}'+\veps_{n}p_{n-1}, \\
q_{n}=2q_{n}'+\veps_{n}q_{n-1}
\end{cases}
\quad\text{ for }n\ge 1.\end{equation}
%By the above recursive formulas, we further have
Further,
\begin{equation}\label{Eq:p_n'+p_n''}
\begin{cases}
p_n=p_n'+p_n'', \\
q_n=q_n'+q_n''
\end{cases}
\text{ and }\quad
\begin{cases}
p_{n-1} = \veps_n(p_n''-p_n'), \\
q_{n-1} = \veps_n(q_n''-q_n')
\end{cases}
\quad \text{ for }n\ge 1.
\end{equation}
%\end{lemma}
%\begin{proof}
%The lemma is obvious if we plug $\mathfrak a_{0}=1$, $\mathfrak b_0=-1$, $\mathfrak a_{2n}=2$, $\mathfrak a_{2n-1}=a_n$, $\mathfrak b_{2n}=-1$, $\mathfrak b_{2n-1}=\veps_n$ %for $n\in\bN$
%in the general continued fraction form in Lemma~\ref{Lem:generalrecursive}.
%\end{proof}
Moreover, letting $p_{-1}=-1, q_{-1}=1$ and $\veps_0=1$, we have the recursive formulas for the principal convergents
%The recursive formulas for the principle convergents are given by the following lemma.
%\begin{lemma}We have 
\begin{equation}\label{Eq:recursivePC} \begin{cases}
p_n=(2a_n+\veps_n-1)p_{n-1}+\veps_{n-1}p_{n-2}, \\
q_n=(2a_n+\veps_n-1)q_{n-1}+\veps_{n-1}q_{n-2}
\end{cases}
\quad \text{ for }n\ge 1.
\end{equation}
%where $p_{-1}/q_{-1}=-1/1$, $p_0/q_0=1/1$ and .
\end{lemma}

Note that $2a_{n}+\veps_{n}-1\ge 2$ for all $n\in\bN$. 
Since $q_1\ge 2q_0+\veps_0q_{-1}=3>q_0$, by \eqref{Eq:recursivePC}, we have 
\begin{align}\label{mono-qn}
q_{n+1}> q_n, \quad \forall n\in \bN.  %the denominators of principal convergents are increasing.
\end{align}
By \eqref{eq:recursivematrix} and the second assertion of \eqref{Eq:p_n'+p_n''}, we have
\begin{equation*}\label{eq:p_n-1}
A_{(a_1,\veps_1)}A_{(a_2,\veps_2)}\cdots A_{(a_n,\veps_n)} \begin{pmatrix} -1 & 1 \\ 1 & 1 \end{pmatrix} = \begin{pmatrix} \veps_np_{n-1} & p_n \\ \veps_nq_{n-1} & q_n \end{pmatrix} \quad \text{ for all }n\in\bN.
\end{equation*}
Thus, by noting $\text{det}(A_{(a_i,\veps_i)}) =\veps_i$, we have 
\begin{equation}\label{eq:det}
p_{n-1}q_n-p_nq_{n-1}=-2\veps_1\cdots \veps_{n-1} \quad \text{ for all }n\in\bN.
\end{equation}

Denote by
$$\zeta_n :=\TO^{n-1}(x)= 1- \cfrac{1}{a_n+ \cfrac{\veps_n}{2-\cfrac{1}{a_{n+1}+\cfrac{\veps_{n+1}}{\ddots}}}}.$$
We can show inductively that
\begin{equation}\label{Eq:completequotient}x=\cfrac{p_{n}''+p_{n}'\zeta_{n+1}}{q_{n}''+q_{n}'\zeta_{n+1}}  \quad \text{ for all }n\in\bN.\end{equation}
%Since 
%\begin{align*}
%\left\{\begin{array}{lll}
%p_{n-1}''+p_{n-1}'=p_{n-1}, \\
%p_{n-1}''-p_{n-1}'=\veps_{n-1}p_{n-2},
%\end{array}\right.
%\quad\text{and}\quad
%\left\{\begin{array}{lll}
%q_{n-1}''+q_{n-1}'=q_{n-1}, \\
%q_{n-1}''-q_{n-1}'=\veps_{n-1}q_{n-2},
%\end{array}\right.
%\end{align*}
By \eqref{Eq:p_n'+p_n''}, we have
\begin{equation*}\begin{cases}
p_{n}''=\frac{1}{2}(p_{n}+\veps_{n}p_{n-1}), \\
p_{n}'=\frac{1}{2}(p_{n}-\veps_{n}p_{n-1}),
\end{cases}
\quad\text{and}\quad\quad
\begin{cases}
q_{n}''=\frac{1}{2}(q_{n}+\veps_{n}q_{n-1}), \\
q_{n}'=\frac{1}{2}(q_{n}-\veps_{n}q_{n-1}),
\end{cases}
\quad \text{ for all }n\in\bN.
\end{equation*}
Thus,
\begin{equation}\label{Eq:completequotient2}
x=\frac{p_{n}(1+\zeta_{n+1})+\veps_{n}p_{n-1}(1-\zeta_{n+1})}{q_{n}(1+\zeta_{n+1})+\veps_{n}q_{n-1}(1-\zeta_{n+1})}.
\end{equation}
The following theorem gives the convergence of our OOCF.
\begin{theorem}\label{Thm:convergence}
For all $x\in[0,1]$, the OOCF expansion of $x$ converges to $x$.
\end{theorem}

\begin{proof}
Let $\xi_n = {1-\zeta_n \over 1+ \zeta_n}$.
By \eqref{Eq:completequotient2} and \eqref{eq:det}, we have
\begin{equation}\label{Eq:pfconv}x - \frac{p_n}{q_n}  
= \frac{\xi_{n+1}\veps_n(p_{n-1}q_n-q_{n-1}p_n)}{q_n(q_{n}+\veps_{n}q_{n-1}\xi_{n+1})} =\frac{-2\xi_{n+1}\veps_1\cdots\veps_n}{q_n(q_{n}+\veps_{n}q_{n-1}\xi_{n+1})}.\end{equation}
%$$x = \frac{p_n+\veps_np_{n-1}\zeta_{n+1}}{q_n+\veps_nq_{n-1}\zeta_{n+1}}.$$
Since $|\veps_{n}|=1$ and $|\xi_{n+1}|\leq 1$,  by \eqref{mono-qn}, we have % by \eqref{eq:det},
$$
\left|x - \frac{p_n}{q_n}\right|<\frac{2}{q_n}\quad \forr \ \all \ n\in\bN.$$
Again by \eqref{mono-qn}, %By \eqref{eq:det}, \eqref{mono-qn}, ... {\color{red} rewrite the proof, even the review's proof is more clear.}
$q_n\to \infty$ as $n\to \infty$, which concludes $p_n/q_n\to x$ as $n\to \infty$.
%Let $x$ be a number whose OOCF expansion is infinite.
%By Lemma~\ref{Lem:relation_of_convergents} (\ref{Lem:rel_conv_3}), the intervals $[p_n/q_n, p''_n/q''_n]$ are shrinking. Thus by Lemma~\ref{Lem:xbtw}, 
%%$$\left|\dfrac{p_n}{q_n}-\dfrac{p_n''}{q_n''}\right|\rightarrow 0 \quad\text{ as } n\rightarrow \infty.$$
%\begin{equation*} 
%\left|x-\frac{p_n}{q_n}\right|<\left|\dfrac{p_n}{q_n}-\dfrac{p_n''}{q_n''}\right| \to 0 \quad \text{as}\quad n\to\infty.
%\qedhere
%\end{equation*}
\end{proof}

\begin{lemma}\label{Lem:xbtw}
For all $x\in(0,1)$,  $x$ is between $p_n/q_n$ and $p''_n/q''_n$.
\end{lemma}

%{\color{red}($\downarrow$ Seulbee: I wrote the following proof of the Lemma according to the reviewer's comment no. 26 and 34. Actually the reviewer commented that the lemma does not need proof by (3.3) but I think we need the following arguments. Is it ok?)} {\color{blue} Tell the reviewer that for the reader's convenience, we add the short proof. {$\rightarrow$ \color{red}Lee: Ok. I have deleted the previous proof.}}
\begin{proof}
The equations \eqref{eq:recursivematrix} and \eqref{eq:recursivematrix2} imply that $p_n'q_n''-p_n''q_n' = \veps_1\cdots \veps_n.$
Thus, by \eqref{Eq:completequotient},
$$x-\frac{p_n''}{q_n''} = \frac{\zeta_{n+1}(p_n'q_n''-p_n''q_n')}{q_n''(q_n''+q_n'\zeta_{n+1})} = \frac{\zeta_{n+1}\veps_1\cdots\veps_n}{q_n''(q_n''+q_n'\zeta_{n+1})}, $$
which means that $x-p_n/q_n$ and $x-p''_n/q''_n$ has opposite signs by comparing with \eqref{Eq:pfconv}.
\end{proof}
%{\color{red}($\downarrow$ Seulbee: This is the previous proof of Lemma 3.5. Which proof is better?)} {\color{blue} We can follow the reviewer's proof. Please delete this proof.}
%\begin{proof}
%%For $n\ge 1$, if $\frac{k-1}{k} \le T^{n-1}_\textrm{OOCF}(x) \le \frac{2k -1}{2k+1}$, then $a_n=k+1$ and $\veps_n=-1$ and
%%if $\frac{2k -1}{2k+1} \le T^{n-1}_\textrm{OOCF}(x) \le \frac{k}{k+1}$, then $a_n=k$ and $\veps_n=1$.
%By (\ref{Eq:letter}), we deduce that $1-T^{n-1}_\textrm{OOCF}(x)$ is between $\frac{1}{a_n+(\veps_n/2)}$ and $\frac{1}{a_n+\varepsilon_n}$.
%%For each OOCF digit $(a,\veps)$, let $f_{(a,\veps)}$ be the inverse of the restriction of $\TO$ to the cell corresponds to 
%%Since $f_{(a,\veps)}$ is a composition of translations and inversions, $f_{(a,\veps)}$ is monotone.
%Let $g = f_{(a_1,\veps_1)}\circ\cdots\circ f_{(a_n,\veps_n)}$.
%Since $g$ is monotone, $g$ does not change the relative positions of points.
%Thus $x=g(1-\TO^{n-1}(x)) $ is between $g(\frac{1}{a_n+(\veps_n/2)}) = p_n/q_n$ and $g(\frac{1}{a_n+\veps_n}) = p_n''/q_n''$.
%%\sout{$x$ is between $p_n/q_n$ and $p_n''/q_n''$.}
%\end{proof}

\medskip
With the above preparations, we can now show the following lemma.
\begin{lemma}\label{Lem:relation_of_convergents}
The following statements hold.
\begin{enumerate}
\item\label{Lem:rel_conv_1} The $n$-th principal convergent $p_{n}/q_{n}$ is between $p'_n/q'_n$ and $p''_n/q''_n$.
\item\label{Lem:rel_conv_2} The $(n-1)$-th principal convergent $p_{n-1}/q_{n-1}$ is not between $p'_{n}/q'_{n}$ and $p''_{n}/q''_{n}$. 
\item\label{Lem:rel_conv_3} The three distinct convergents $p_n/q_n$, $p_{n}'/q_{n}'$ and $p_{n}''/q_{n}''$ are in the half closed interval $I_{n-1}$ of endpoints $p_{n-1}/q_{n-1}$ and $p_{n-1}''/q_{n-1}''$ which contains $p_{n-1}''/q_{n-1}''$ but does not contain $p_{n-1}/q_{n-1}$. % and at least one of $p_{n+1}'/q_{n+1}'$ and $p_{n+1}''/q_{n+1}''$ is not $p_n'/q_n'$ and $p_n''/q_n''$.
\end{enumerate}
\end{lemma}

\begin{proof}
The first two assertions follow from \eqref{Eq:p_n'+p_n''} and the fact that for two rationals $a/b$ and $c/d$ such that $bd>0$, if $a/b \le c/d$, then
$$\frac{a}{b}\le \frac{a+c}{b+d}\le \frac{c}{d}.$$
%Then, the first and second assertions .
%By Lemma~\ref{Lem:recursive}, $p_n=p_n'+p_n''$ and $q_n=q_n'+q_n''$.
%Hence, $p_{n}/q_{n}$ is between $p'_n/q'_n$ and $p''_n/q''_n$.

%(2) 
%On the Farey graph, 
%there are only two triangles which share the arc connecting $p_{n}'/q_{n}'$ and $p_{n}''/q_{n}''$.
%By Lemma~\ref{Lem:adjacency_of_convergents}, the other end points of the two triangles are $p_{n-1}/q_{n-1}$ and $p_{n}/q_{n}$.
%By (\ref{Lem:rel_conv_1}), $p_{n-1}/q_{n-1}$ can not be between $p'_n/q'_n$ and $p''_n/q''_n$ since two ideal triangles of the Farey graph do not overlap each other.

For (3), by Lemma~\ref{Lem:recursive}, we have
\begin{equation*}
\begin{cases}
p_{n}'=(a_{n}-1)p_{n-1}+p_{n-1}'',\\
q_{n}'=(a_{n}-1)q_{n-1}+q_{n-1}'',
\end{cases}
\quad \text{and} \quad
\begin{cases}
p_{n}''=(a_{n}-1+\veps_{n})p_{n-1}+p_{n-1}'', \\ 
q_{n}''=(a_{n}-1+\veps_{n})q_{n-1}+q_{n-1}''.
\end{cases}
\end{equation*}
Since $a_{n}-1\ge 0$ and $a_{n}+\veps_{n}-1\ge 0$, both $p_{n}'/q_{n}'$ and $p_{n}''/q''_{n}$ are in $I_{n-1}$.
By the first assertion, % (\ref{Lem:rel_conv_1}), 
$p_n/q_n$ is also in $I_{n-1}$.
\end{proof}

%\les{The following theorem is an analog of the theorem of Lagrange and Euler for the RCFs of quadratic irrationals (See \cite[Chapter III-\S1]{RoSz92} and \cite[Section 10]{Khi63}).}
%
%\begin{theorem}
%An eventually periodic OOCF expansion converges to an $\infty$-rational or a quadratic irrational.
%Moreover, a quadratic irrational has an eventually periodic OOCF expansion.
%\end{theorem}
\medskip
Now, we are ready to prove Theorem \ref{Thm:main2}.
\begin{proof}[Proof of Theorem \ref{Thm:main2}]
%\noindent \eqref{Prop:1-rat:irr}:
If $x$ has an eventually periodic OOCF, then there exist distinct positive integers $i$ and $j$ such that $T^i(x)=T^j(x)$.
Since $T^i$ and $T^j$ are linear fractional maps, $T^i(x)$ is either $0$ or a quadratic irrational.
In the former case $x$ is an $\infty$-rational, while in the latter case $x$ is a quadratic irrational.
%If $x$ has an eventually periodic OOCF, then there exist $i$ and $j$ such that $\zeta_{i+1}=\zeta_{j+1}$.
%By (\ref{Eq:completequotient}), 
%\[
%x=\cfrac{p_{i}''+p_{i}'\zeta_{i+1}}{q_{i}''+q_{i}'\zeta_{i+1}}=\cfrac{p_{j}''+p_{j}'\zeta_{j+1}}{q_{j}''+q_{j}'\zeta_{j+1}}.
%\]
%Thus we have $x=p_i''/q_i''$ or 
%$${1 \over \zeta_{i+1}}=\cfrac{p_i'-q_i'x}{q_i''x-p_i''}=\cfrac{p_j'-q_j'x}{q_j''x-p_j''}={1 \over \zeta_{j+1}}.$$
%The first case corresponds to the case that $x$ is an $\infty$-rational and $\zeta_{i+1}=0$. 
%While the second case implies 
%$$(q_i'q_j''-q_j'q_i'')x^2+(q_j'p_i''+p_j'q_i''-q_i'p_j''-p_i'q_j'')x+(p_i'p_j''-p_i''p_j')=0,$$
%which means $x$ is a quadratic irrational (remark that $q_i'q_j''-q_j'q_i''\neq 0$).
%Since $q_j'$ and $q_j$ are relatively prime, if $q_i'q_j=q_j'q_i$, then $q_j|q_i$.
%It contradicts to $q_j>q_i$.

%\noindent (5):
For the second assertion, let $x$ be a quadratic irrational between $0$ and $1$ such that $\alpha_1x^2+\beta_1x+\gamma_1=0$ where $\alpha_1$, $\beta_1$ and $\gamma_1$ are coprime integers.
%Since $$x=\cfrac{p_i''+p_i'\zeta_{i+1}}{q_i''+q_i'\zeta_{i+1}},$$
By  (\ref{Eq:completequotient}), we have for all $i\geq 1$,
$$\alpha_1(p_i''+p_i'\zeta_{i+1})^2+\beta_1(p_i''+p_i'\zeta_{i+1})(q_i''+q_i'\zeta_{i+1})+\gamma_1(q_i''+q_i'\zeta_{i+1})^2=0.$$
%Then
%$$(a_1p_i'^2+b_1p_i'q_i'+c_1q_i'^2)\zeta_{i+1}^2+(2a_ip_i''p_i'+b_1(p_i''q_i'+p_i'q_i'')+2c_1q_i''q_i')\zeta_{i+1}+(a_1p_i''^2+b_1p_i''q_i''+c_1q_i''^2)=0.$$
For $i\geq 1$, let
\begin{equation}\begin{split}\label{Eq:coefficients}
&\alpha_{i+1}=\alpha_1(p_i')^2+\beta_1p_i'q_i'+\gamma_1(q_i')^2,\\
&\beta_{i+1}=2\alpha_1p_i''p_i'+\beta_1(p_i''q_i'+p_i'q_i'')+2\gamma_1q_i''q_i',\\% \text{ and}\\
&\gamma_{i+1}=\alpha_1(p_i'')^2+\beta_1p_i''q_i''+\gamma_1(q_i'')^2.
\end{split}\end{equation}
Then 
\begin{align}\label{Eq:zetas}
\alpha_{i+1}\zeta_{i+1}^2+\beta_{i+1}\zeta_{i+1}+\gamma_{i+1}=0.
\end{align}
Since $|q_i'p_i''-q_i''p_i'|=1$, we can check that
\begin{align}\label{bounded-coeff}
\beta_{i+1}^2-4\alpha_{i+1}\gamma_{i+1}=\beta_1^2-4\alpha_1\gamma_1.
\end{align}
On the other hand, we also have
%$$|q_i''(q_i'x-p_i')+q_i'(p_i''-q_i''x)|=1.$$
$|(x-p_i'/q_i')+(p_i''/q_i''-x)|=1/q_i'q_i''.$
%{\color{blue} (SB: I added the explanation:)}
Thus, by Lemma \ref{Lem:xbtw} and the fact that $p_i/q_i$ is between $p_i'/q_i'$ and $p_i''/q_i''$, we deduce that $x-p_i'/q_i'$ and $x-p_i''/q_i''$ have opposite signs.
%Thus, $q_i''|q_i'x-p_i'|<1$ and $q_i'|p_i''-q_i'x|<1.$ %{\color{red} ( why?)}
Thus, $|q_i'x-p_i'|<1/q_i''$ and $|p_i''-q_i''x|<1/q_i'.$
Hence, there are $|\delta|<1$ and $|\lambda|<1$ such that
\begin{equation}\label{Eq:p_i'p_i''}
p_i' = q_i'x+\delta/q_i'' \ \text{ and } \ p_i'' = q_i''x+\lambda/q_i'.
\end{equation}
By plugging \eqref{Eq:p_i'p_i''} in \eqref{Eq:coefficients}, we derive the following expressions 
\begin{equation}\label{eq:a}
\alpha_{i+1}
%=a_1\left(q_i'x+\frac{\alpha}{q_i''}\right)^2+b_1\left(q_i'x+\frac{\alpha}{q_i''}\right)q_i'+c_1(q_i')^2 
=\delta\left(\frac{q_i'}{q_i''}(2\alpha_1x+\beta_1)+\alpha_1\frac{\delta}{(q_i'')^2}\right), 
\end{equation}
\begin{equation}\label{eq:b}
\beta_{i+1} 
%& =2a_1\left(q_i''x+\frac{\beta}{q_i'}\right)\left(q_i'x+\frac{\alpha}{q_i''}\right)+b_1\left(\left(q_i''x+\frac{\beta}{q_i'}\right)q_i'+\left(q_i'x+\frac{\alpha}{q_i''}\right)q_i''\right)+2c_1q_i'q_i''  \\ & 
= (2\alpha_1x+\beta_1)(\delta+\lambda)+2\alpha_1\frac{\delta\lambda}{q_i'q_i''}  \text{ and }
\end{equation}
\begin{equation}\label{eq:c}
\gamma_{i+1}
%=a_1\left(q_i''x+\frac{\beta}{q_i'}\right)^2+b_1\left(q_i''x+\frac{\beta}{q_i'}\right)q_i''+c_1(q_i'')^2 
=\lambda\left(\frac{q_i''}{q_i'}(2\alpha_1x+\beta_1)+a_1\frac{\lambda}{(q_i')^2}\right). %\text{ and} 
\end{equation}

By \eqref{eq:b}, we have $|\beta_{i+1}| \le 2(|2\alpha_1|+|\beta_1|)+|2\alpha_1|$, thus the coefficient $\beta_{i+1}$ is bounded.
%\sout{
%By  \eqref{eq:a}, we have $|\alpha_{i+1}|<2|\alpha_1|+|\beta_1|+|\alpha_1|$ \don{(do we need $q_i''\ge q_i'$ here?)}.
%If $q_i''\ge q_i'$, then $\alpha_{i+1}$ is bounded. % since $|a_{i+1}|<2|a_1|+|b_1|+|a_1|$ by \eqref{eq:abc}.
%}
If $q_i''\ge q_i'$, then by \eqref{eq:a}, we have $|\alpha_{i+1}|<2|\alpha_1|+|\beta_1|+|\alpha_1|$.
Thus, $\alpha_{i+1}$ is bounded. %{\color{red}($\leftarrow$ Seulbee: Is it ok?)}% since $|a_{i+1}|<2|a_1|+|b_1|+|a_1|$ by \eqref{eq:abc}.
Further, by (\ref{bounded-coeff}), $\gamma_{i+1}$ is bounded.
Similarly, if $q_i''<q_i'$, then by \eqref{eq:c}, $\gamma_{i+1}$ is bounded since $|\gamma_{i+1}|<2|\alpha_1|+|\beta_1|+|\alpha_1|$. Moreover, by (\ref{bounded-coeff}), $\alpha_{i+1}$ is also bounded.
%If $q_i''>q_i'$, then $|q_i'x-p_i'|<1/q_i''<1/q_i'$.
%Then, there is $\veps$ such that $p_i'=q_i'x+\veps/q_i'$ where $|\veps|<1$.
%Hence, 
%$$a_{i+1}=a_1(q_i'x+\veps/q_i')^2+b_1(q_i'x+\veps/q_i')q_i'+c_1q_i'^2 =\veps(2a_1x+b_1+a_1\veps/q_i^2).
%$$
%The coefficient $a_{i+1}$ is bounded since $|a_{i+1}|<2|a_1x|+|b_1|+|a_1|$.
%If $q_i''<q_i'$, then $|q_i''x-p_i''|<1/q_i'<1/q_i''.$
%Similarly, we can show that $c_{i+1}$ is bounded.({\color{red}why?})
%By (\ref{bounded-coeff}), $b_{i+1}$ is also bounded. 
%Thus the coefficients of all the equations of $\zeta_{i+1}$ are bounded.
Thus in all cases, the coefficients of the equation \eqref{Eq:zetas} are all bounded. Therefore, $\{\zeta_i\}_{i\in\bN}$ has only finitely many values which means that $\zeta_n=\zeta_m$ for some $m$ and $n$.
Therefore, the OOCF expansion of $x$ is eventually periodic.
\end{proof}

\begin{remark}
From the proof of Proposition \ref{Thm:1-rat_quad}-\eqref{Prop:1-rat:inftyrat}, we see that if $x$ is an $\infty$-rational, then its OOCF ends with $(2,-1)^\infty$ and thus there exists $n_0\geq 0$ such that $\zeta_{n+1}=0$ for all $n\geq n_0$. Hence, by (\ref{Eq:completequotient}), we have $x=p''_n/q''_n$ for all $n\geq n_0$.
\end{remark}

At the end of this section, let us discuss the relation between the OOCF convergents of a number $x$ and the EICF convergents of $1-x$.

We denote the EICF expansion in \eqref{Eq:eicf} by a sequence in a double angle bracket:
$$%\cfrac{1}{b_1+\cfrac{\eta_1}{b_2+\cfrac{\eta_2}{b_3+\cfrac{\eta_3}{\ddots}}}}=
\ldab(b_1,\eta_1),(b_2,\eta_2),\cdots, (b_n,\eta_n),\cdots \rdab=\ldab(b_n,\eta_n)_{n\in\bN}\rdab.
$$
%We denote by $p_n^E(x)/q_n^E(x)$ the $n$-th EICF convergent of $x$.
%\sout{The map $\TE$ is the left shift map of EICF expansions, i.e.,}
%$$\TE(\ldab(b_1,\eta_1),(b_2,\eta_2),\cdots, (b_n,\eta_n),\cdots \rdab) = \ldab(b_2,\eta_2),(b_3,\eta_3),\cdots, (b_n,\eta_n),\cdots \rdab.$$
The $n$-th EICF convergent is denoted by 
$$\frac{p^E_n}{q^E_n}=\ldab(b_1,\eta_1),(b_2,\eta_2),\cdots, (b_n,\eta_n)\rdab.$$
By Lemma \ref{Lem:generalrecursive}, we have the following matrix relation:
\begin{equation}\label{eicf}
\begin{pmatrix}
b_0 & \eta_0 \\ 1 & 0
\end{pmatrix}
\begin{pmatrix}
b_1 & \eta_1 \\ 1 & 0
\end{pmatrix}
\cdots 
\begin{pmatrix}
b_n & \eta_n \\ 1 & 0
\end{pmatrix}
= 
\begin{pmatrix}
q^E_n & \eta_n q^E_{n-1} \\ p^E_n & \eta_n p^E_{n-1}
\end{pmatrix}.
\end{equation}
Since each matrix in \eqref{eicf} belongs to $\Theta\cup\begin{pmatrix}0&1\\1&0\end{pmatrix}\Theta$, each EICF convergent $p^E_n/q^E_n$ is an $\infty$-rational.

Observe that if $p_n^E(x)/q_n^E(x)$ is of type even/odd, then $1-p_n^E(x)/q_n^E(x)$ is $1$-rational.
If $p_n^E(x)/q_n^E(x)$ is of type odd/even, then $1-p_n^E(x)/q_n^E(x)$ is still of type odd/even.
%The set of $1-p_n^E(1-x)/q_n^E(1-x)$ contains the best $1$-rational approximations of $x$.
%However, the set contains the other rationals. 

\begin{proposition}
Let $x\in(0,1)$.
All rationals of type odd/odd in $\{1-p_n^E(1-x)/q_n^E(1-x): n\ge 1 \}$ are best $1$-rational approximations of $x$, and hence are OOCF principal convergents of $x$.
%Moreover, the sequence of odd/odd in $1-p_n^E(1-x)/q_n^E(1-x)$ is the subsequence of OOCF principal convergents of $x$.
\end{proposition}

\begin{proof}
For each $n\geq 0$, denote by $P_n/Q_n:=1-p_n^E(1-x)/q_n^E(1-x)$.
We have
$$P_n = q_n^E(1-x)-p_n^E(1-x) \ \andd \ Q_n = q_n^E(1-x).$$
By the theorem of Short and Walker (see \eqref{Eq:SW}; also \cite[Theorem 5]{ShWa14}), 
for any $a/b\in\Theta(\infty)$ such that $1\le b\le q_n^E(1-x)$ and $a/b\not=p_n^E(1-x)/q_n^E(1-x)$,  
$$|P_n-Q_nx |= |q_n^E(1-x)-p_n^E(1-x)-x \cdot q_n^E(1-x)| 
%= |(1-x)\cdot q_n^E(1-x)-p_n^E(1-x)|
 < |(1-x)b-a|.$$ 
For any $c/d\in\Theta(1)$ such that $1\le d\le Q_n$ and $c/d\not=P_n/Q_n$, we have 
$${d-c \over d} = 1-{c \over d}\in\Theta(\infty) \ \andd \ {d-c \over d} \not=p_n^E(1-x)/q_n^E(1-x).$$
Thus $|P_n-Q_nx|<|(1-x)d-(d-c)|=|c-dx|$, which means that $P_n/Q_n$ is a best $1$-rational approximation of $x$.
\end{proof}

The next proposition describes a connection between the principal convergents of OOCF and EICF. 
Recall that $f(x) = \frac{1-x}{1+x}$ is the conjugacy map defined in Section 2.
\begin{proposition}
%Let $f(x) = \frac{1-x}{1+x}$.
Let $x\in(0,1)$. There is a 1-1 correspondence between the partial quotients of OOCF of $x$ and that of EICF of $f(x)$. In particular, $p_n^E(f(x))/q_n^E(f(x))=f(p_n(x)/q_n(x))$
{ for all $n \ge 1$.}
%{\color{teal}Moreover, .}
%The EICF principal convergents of $f(x)$ is $f(\text{The OOCF principal convergents of }x)$.
\end{proposition}

\begin{proof}
Since
\begin{equation*}%\label{Eq:cells}
f(B(k+1,-1))=\left[\frac{1}{2k},\frac{1}{2k-1}\right],\quad f(B(k,1))=\left[\frac{1}{2k+1},\frac{1}{2k}\right],
\end{equation*}
%Thus, $f(B(k+1,-1))$, $f(B(k,1))$ are subintervals which correspond to $(2k,1), (2k,-1)\in\mathcal{A}_{\text{EICF}}$.
%Thus, if $x\in B(k+1,-1)\cup B(k.1)$, then $\frac{p_1}{q_1}=\frac{2k-1}{2k+1}$ and $\frac{p_n^E(f(x))}{q_n^E(f(x))}=\frac{1}{2k}$.
%In other words, 
there is a 1-1 correspondence $\phi$ between the partial quotients of OOCF and EICF as follows:
%$$\phi: (k+1,-1)\mapsto (2k,1), \quad (k,1)\mapsto (2k,-1).$$
$$\phi: (k+1,-1)\mapsto (2k,-1), \quad (k,1)\mapsto (2k,1).$$
Thus, for $x=[\![(a_n,\veps_n)_{n\in\bN}]\!]$, the EICF expansion of $f(x)$ is $\ldab \phi(a_n,\veps_n)_{n\in\bN}\rdab$. Considering the finite expansions, we obtain the second assertion. 
\end{proof}
%{\color{red} The proof is a more general one which shows the correspondance of the two expansions. The proposition is only for finite expansions. {\color{teal}$\uparrow$Lee: Please check the second sentence in the statement.}}

\section{Best $1$-rational approximation}\label{Sec:BestA}
\begin{figure}[t]
\begin{tikzpicture}[scale=2.5]
\filldraw[gray!40] (0,2.5) -- (-2.5,2.5) arc (180:270:2.5) -- cycle;
\draw (0,0) arc (-90:-180:2.5);
%\draw (-5,2.5) -- (-2.5,2.5) arc (0:-90:2.5) -- cycle;
\draw (-5,0) arc (-90:0:2.5);
\draw (-2.5,0) arc (-90:270:0.625);
\filldraw[fill = gray!40] (-3.3333,0) arc (-90:270:0.27777);
\draw (-1.6666,0) arc (-90:270:0.27777);
\draw(-0.75*5,0) arc (-90:270:0.15625);
\draw(-0.25*5,0) arc (-90:270:0.15625);
\filldraw[fill = gray!40](-0.8*5,0) arc (-90:270:0.1);
\draw(-0.6*5,0) arc (-90:270:0.1);
\filldraw[fill = gray!40](-0.4*5,0) arc (-90:270:0.1);
\draw (-0.2*5,0) arc (-90:270:0.1);

	\node at (-5,-0.2)  {$\frac01$};	\node at (0,-0.2)  {$\frac11$};	
	\node at (-2.5,-0.2)  {$\frac 12$};	
	\node at (-3.33333,-0.2)  {$\frac 13$};	\node at (-1.6666,-0.2)  {$\frac 23$};	
	\node at (-3.75,-0.2){$\frac 14$}; \node at (-3,-0.2){$\frac 25$};	
	\node at (-2,-0.2){$\frac35$}; \node at (-1.25,-0.2){$\frac34$};
	\node at (-1,-0.2){$\frac 45$}; 
	\node at (-0.8*5,-0.2){$\frac 15$};
	
	\draw (-5.5,0) -- (0.5,0);
	\draw (0,-0.05) -- (0,0.05);
	\draw (-5,-0.05) -- (-5,0.05);

%1/6
\draw(-5+0.16666*5,0) arc (-90:270:0.01388888*5); \node at (-5*0.16666*5,-0.2) {$\frac 16$};
%1/7
\filldraw[fill = gray!40](-5+5/7,0) arc (-90:270:5/98); \node at (-5+5/7,-0.2) {$\frac 17$};
%1/8
\filldraw[fill = white](-5+5/8,0) arc (-90:270:5/128); \node at (-5+5/8,-0.2) {$\frac 18$};
%1/9
\filldraw[fill = gray!40](-5+5/9,0) arc (-90:270:5/81*0.5); \node at (-5+5/9,-0.2) {$\frac 19$};
%1/10
\filldraw[fill = white](-5+5/10,0) arc (-90:270:5*1/10*1/10*0.5); %\node at (-5+5/10,-0.2) {$\frac {1}{10}$};

%2/9
%\filldraw[fill = white](-5+5*2/9,0) arc (-90:270:5*1/9*1/9*0.5); \node at (-5+5*2/9,-0.2) {$\frac {2}{9}$};
\hocir{2}{9}{white}
\hocir{2}{7}{white}
\hocir{3}{8}{white}
\hocir{3}{7}{gray!40}
\hocir{5}{9}{gray!40}
\hocir{4}{9}{white}
\hocir{4}{7}{white}
\hocir{5}{8}{white}
\hocir{5}{7}{gray!40}
\hocir{7}{9}{gray!40}
\hocir{5}{6}{white}
\hocir{6}{7}{white}
\hocir{7}{8}{white}
\hocir{8}{9}{white}
\filldraw[fill = gray!40](-5+5*9/11,0) arc (-90:270:5*1/11*1/11*0.5);
\filldraw[fill = white](-5+5*3/10,0) arc (-90:270:5*1/10*1/10*0.5);

%\filldraw[fill = #3](-5+5*#1/#2,0) arc (-90:270:5*1/#2*1/#2*0.5); \node at (-5+5*#1/#2,-0.2) {$\frac {#1}{#2}$};

\end{tikzpicture}
\caption{Ford circles: white circles are based at $\infty$-rationals and gray circles are based at $1$-rationals}
\label{Fig:fordcircles}
\end{figure}
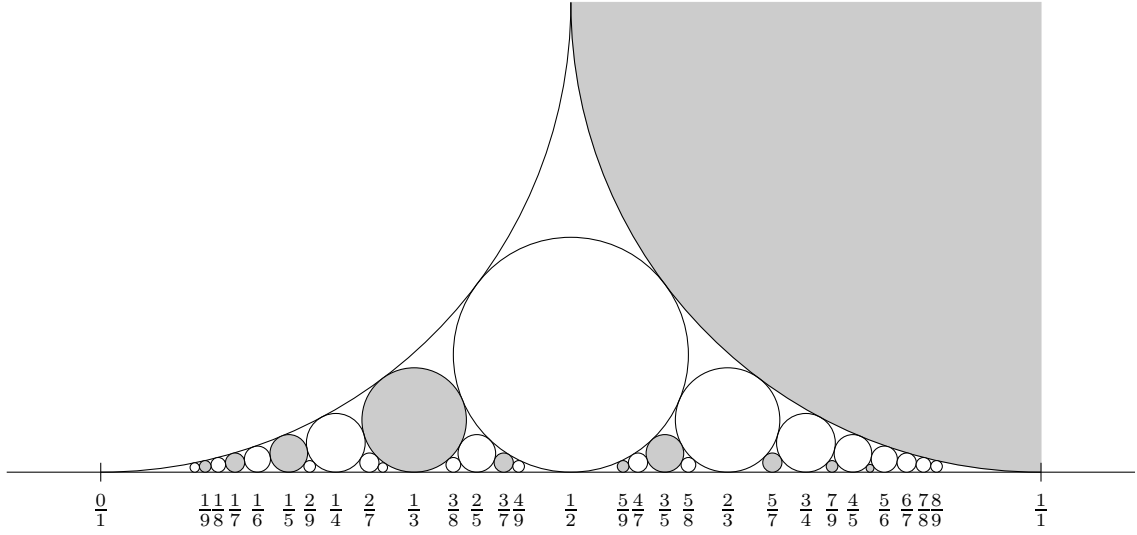
In this section, we prove Theorem~\ref{Thm:main1}.
Recall that $\bH$ is the upper half-plane.
%Now, let us consider the hyperbolic plane $\bH$ as the upper half-plane model.
The boundary of $\bH$ is $\bR_{\infty}=\bR\cup\{\infty\}$.
Denote by $C_{a/b}$ the horocycle of $\bH$ based at $a/b$ whose Euclidean radius is $(2b^2)^{-1}$ and $C_\infty$ the line $\{z=x+i\in\bH:x\in\bR\}$ (see Figure~\ref{Fig:fordcircles}).
We call $C_{a/b}$ \emph{a Ford circle}. The radius of $C_{a/b}$ is denoted by $\text{rad}(C_{a/b})$.
We remark that two Ford circles $C_{a/b}$ and $C_{c/d}$ are adjacent to each other if and only if $|ad-bc|=1.$
Let $R_{a/b}(x)$ be the Euclidean radius of the horocycle based at $x$ tangent to $C_{a/b}$. Then by Pythagorean theorem,  $$R_{a/b}(x)=\frac{1}{2}|bx-a|^2.$$
Thus
\begin{align}\label{equivalence:Ford}
|qx-p|<|bx-a| \Longleftrightarrow R_{p/q}(x)<R_{a/b}(x).
\end{align}
Since the Ford circles are not overlapped each other, we have
\begin{equation}\label{Eq:pf5}\text{rad}(C_{a/b})\le R_{c/d}\left( a/b\right)\text{ for all } a/b,\ c/d\in\bQ.\end{equation}

%
%\begin{theorem}[\les{(Seulbee: Do we need a rephrase of Theorem~\ref{Thm:main} here?)}]\label{Thm:bestapp}
%All principal convergents $p_n/q_n$ are best $1$-rational approximations, and vice versa.
%\end{theorem}
With these preparations, we are ready to prove our Theorem~\ref{Thm:main1}.
\begin{proof}[Proof of Theorem~\ref{Thm:main1}]

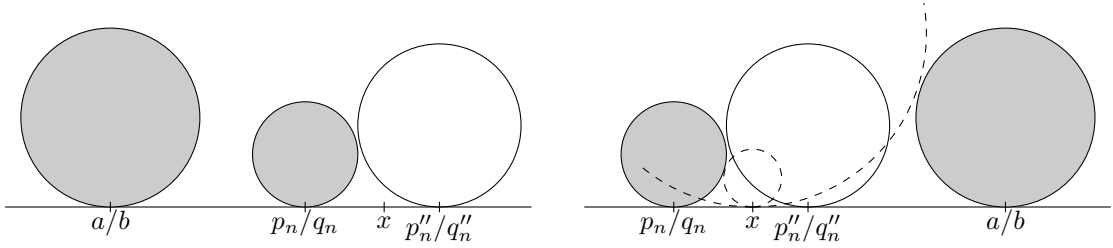
\begin{figure}[t]
\begin{tikzpicture}[scale=0.7]
\draw (-7,0) -- (3,0);
\filldraw[fill = gray!40] (-5,0) arc (-90:270:1.7);
\node at (-5,-0.3) {$a/b$};
\draw (-5,-0.1) -- (-5,0.1);

\filldraw[fill = gray!40] (-0.26*5,0) arc (-90:270:1);
\node at (-0.26*5,-0.3) {$p_n/q_n$};
\draw (-0.26*5,-0.1) -- (-0.26*5,0.1);

\draw (0.25*5,0) arc (-90:270:1.55);
\node at (0.25*5,-0.4) {$p''_n/q''_n$};
\draw (0.25*5,-0.1) -- (0.25*5,0.1);

\node at (0.1*5-0.3,-0.3) {$x$};
\draw (0.1*5-0.3,-0.1) -- (0.1*5-0.3,0.1);

\draw (7-3,0) -- (7+7,0);
\filldraw[fill = gray!40] (7+5,0) arc (-90:270:1.7);
\node at (7+5,-0.3) {$a/b$};
\draw (7+5,-0.1) -- (7+5,0.1);

\filldraw[fill = gray!40] (7-0.26*5,0) arc (-90:270:1);
\node at (7-0.26*5,-0.3) {$p_n/q_n$};
\draw (7-0.26*5,-0.1) -- (7-0.26*5,0.1);

\draw (7+0.25*5,0) arc (-90:270:1.55);
\node at (7+0.25*5,-0.4) {$p''_n/q''_n$};
\draw (7+0.25*5,-0.1) -- (7+0.25*5,0.1);

\draw[style = dashed] (7+0.1*5-0.3,0) arc (-90:270:0.55);
\draw[style = dashed] (7+0.1*5-0.3,0) arc (-90:10:3.3);
\draw[style = dashed] (7+0.1*5-0.3,0) arc (-90:-130:3.3);
\node at (7+0.1*5-0.3,-0.3) {$x$};
\draw (7+0.1*5-0.3,-0.1) -- (7+0.1*5-0.3,0.1);

\end{tikzpicture}
\caption{Two possible relative locations of $x$, $p_n/q_n$, $a/b$ and $p''_n/q''_n$ in the proof of Theorem~\ref{Thm:main1}. The dashed circles are the horocycles based at $x$ tangent to $C_{p_n/q_n}$ and $C_{a/b}$.}
\label{Fg:bestapp_1}
\end{figure}

%\begin{figure}[h]
%\begin{tikzpicture}[scale=0.5]
%\draw (-3,0) -- (7,0);
%\filldraw[fill = gray!40] (5,0) arc (-90:270:1.7);
%\node at (5,-0.3) {$w$};
%\draw (5,-0.1) -- (5,0.1);
%
%\filldraw[fill = gray!40] (-0.26*5,0) arc (-90:270:1.5);
%\node at (-0.26*5,-0.3) {$u$};
%\draw (-0.26*5,-0.1) -- (-0.26*5,0.1);
%
%\draw (0.25*5,0) arc (-90:270:1.05);
%\node at (0.25*5,-0.3) {$v$};
%\draw (0.25*5,-0.1) -- (0.25*5,0.1);
%
%\draw[style = dashed] (0.1*5,0) arc (-90:270:0.55);
%\draw[style = dashed] (0.1*5,0) arc (-90:20:2.95);
%\draw[style = dashed] (0.1*5,0) arc (-90:-130:2.95);
%\node at (0.1*5,-0.3) {$x$};
%\draw (0.1*5,-0.1) -- (0.1*5,0.1);
%
%\end{tikzpicture}
%%\caption{}
%\end{figure}
%We will use geometric properties of Ford circles to prove the theorem.
%Let $R_{a/b}(x)$ be the Euclidean radius of the horocycle based at $x$ which is externally tangent to $C_{a/b}$.
%Then 
%$$\left(R_{a/b}(x)-\frac{1}{2b^2}\right)^2+\left(x-\frac ab\right)^2=\left(R_{a/b}(x)+\frac{1}{2b^2}\right)^2.$$
Given $x\in\bR\setminus\bQ$, let us consider its $n$-th principal convergent $p_n/q_n$ and its $n$-th pseudo-convergent $p''_n/q''_n$. %be the OOCF convergent of $x$. 
Let $a/b\in\Theta(1)$ such that $a/b\not=p_n/q_n$ and $1\le b\le q_n$.
Then 
\begin{equation}\label{Eq:pf1}\text{rad}(C_{p_n/q_n})\le \text{rad}(C_{a/b}).\end{equation}
By \eqref{eq:recursivematrix} and \eqref{eq:recursivematrix2}, the Ford circles $C_{p_n/q_n}$ and $C_{p''_n/q''_n}$ are tangent to each other.
%and Proposition~\ref{Thm:parity_of_convergents}
By Lemma~\ref{Lem:xbtw}, % is an $\infty$-rational  such that  and $x$ is between $p_n/q_n$ and $p''_n/q''_n$. }
 $x$ is between $p_n/q_n$ and $p''_n/q''_n$, then
\begin{equation}\label{Eq:pf3}R_{p_n/q_n}(x)\le \text{rad}(C_{p''_n/q''_n}).\end{equation}
%there is an $\infty$-rational $p''/q''$ which is adjacent to $p/q$ and $x$ is between $p/q$ and $r$.
%Since the radius of $C_w$ is  larger than the radius of $C_u$, $v$ is outside of the interval $[u,v]$.
%If $w$ is on the side of $u$, then $R_w(x)\ge R_u(x)$.
%Assume that $w$ is on the side of $v$.
%For every $z$ between $u$ and $v$, the radius of $C_z$ is at most the radius of $C_u$ since Ford circles do not appear to overlap.
Let $I_n$ be the closed interval of endpoints $p_n/q_n$ and $p''_n/q''_n$.
Since $\text{rad}(C_{r/s})\le \text{rad}(C_{p_n/q_n})$ for any $r/s\in I_n\cap \bQ$, we deduce from \eqref{Eq:pf1} (as shown in Figure \ref{Fg:bestapp_1}) that $a/b\not\in I_n$.
%from $$\text{rad}(C_{a/b})\ge \text{rad}(C_{p_n/q_n}).$$}
%The radius $R_{p/q}(x)$ is at most the radius of $C_r$.
%Since the radius of $C_{a/b}$ is at least the radius of $C_{p/q}$, the $1$-rational ${a/b}$ is outside of the inverval $[p/q,r]$ (as shown in Figure \ref{Fg:bestapp_1}).
Then we have 
\begin{equation}\label{Eq:pf2}\text{rad}(C_{p''_n/q''_n})<R_{a/b}(x).
\end{equation} %is larger than the radius of $C_r$. }
%(The equality holds if and only if $a/b=\infty$.)
By \eqref{Eq:pf3} and \eqref{Eq:pf2}, we have $R_{p_n/q_n}(x) < R_{a/b}(x)$. Hence by \eqref{equivalence:Ford}, $p_n/q_n$ is a best $1$-rational approximation of $x$. 

Conversely, assume that $a/b\in\Theta(1)$ is not a principal convergent of OOCF of $x$.
Then there are consecutive principal convergents $p_{n-1}/q_{n-1}$ and $p_{n}/q_{n}$ such that 
$q_{n-1} \le b < q_{n}$ and $a/b\not=p_{n-1}/q_{n-1}.$
Thus, \begin{equation}\label{Eq:pf4}\text{rad}(C_{{p_{n}}/{q_{n}}})<\text{rad}(C_{a/b}).\end{equation}
%Both are adjacent to $p_{n+1}'/q_{n+1}'$ and $p_{n+1}''/q_{n+1}''$ in the Farey graph, i.e.,
By \eqref{eq:recursivematrix} and \eqref{eq:recursivematrix2}, $C_{{p_{n-1}}/{q_{n-1}}}$ and $C_{{p_{n}}/{q_{n}}}$ are tangent to both $C_{{p_{n}'}/{q_{n}'}}$ and $C_{{p_{n}''}/{q_{n}''}}$.
Without loss of generality, we assume that $p_{n-1}/q_{n-1}<p_{n}'/q_{n}'<p_{n}''/q_{n}''$ (see Figure~\ref{Fg:bestapp_2}).
%\les{Let $I'_n$ be the closed interval between $p'_{n}/q'_{n}$ and $p''_{n}/q''_{n}$.}
%Since \les{$\text{rad}(C_{{p_{n}}/{q_{n}}})<\text{rad}(C_{a/b})$, %$a/b\not\in I'_{n+1}$.}
By \eqref{Eq:pf4}, $a/b\not\in[p'_{n}/q'_{n},p''_{n}/q''_{n}]$.
%$a/b$ is not between $p_{n+1}'/q_{n+1}'$ and $p_{n+1}''/q_{n+1}''$.
Now we will show that $R_{p_{n-1}/q_{n-1}}(x)<R_{a/b}(x)$ which by \eqref{equivalence:Ford}, implies that $a/b$ is not a best $1$-rational approximation of $x$. %, i.e., $a/b$ is not a $1$-rational best approximation.}
%
%\begin{figure}[h]
%\begin{tikzpicture}[scale=1]
%\draw (-2,0) -- (7,0);
%\draw (4.5,0) arc (-90:-20:2.65);
%\draw (4.5,0) arc (-90:-380:2.65);
%\node at (4.5,-0.5) {$\frac {p_{n+1}''}{q_{n+1}''}$};
%\draw (4.5,-0.1) -- (4.5,0.1);
%
%\draw (2.23,0) arc (-90:270:0.48);
%\node at (2.23,-0.5) {$\frac {p_{n+1}'}{q_{n+1}'}$};
%\draw (2.23,-0.1) -- (2.23,0.1);
%
%\filldraw[fill = gray!40] (0.1*5,0) arc (-90:270:1.5);
%\node at (0.1*5,-0.5) {$\frac{p_n}{q_n}$};
%\draw (0.1*5,-0.1) -- (0.1*5,0.1);
%
%\filldraw[fill = gray!40] (2.9,0) arc (-90:270:0.23);
%\node at (3,-0.5) {$\frac {p_{n+1}}{q_{n+1}}$};
%\draw (2.9,-0.1) -- (2.9,0.1);
%
%%\draw (0.25*5,0) arc (-90:270:1.05);
%%\node at (0.25*5,-0.3) {$v$};
%%\draw (0.25*5,-0.1) -- (0.25*5,0.1);
%
%\draw[style = dashed] (3.8,0) arc (-90:270:1.8);
%\draw[style = dashed, very thick] (3.8,0) arc (-90:-140:7.7);
%\node at (-1.9,3) {$C$};
%
%\node at (3.8,-0.4) {$x$};
%\draw (3.8,-0.1) -- (3.8,0.1);
%
%\filldraw[fill = gray!40] (1.6,0) arc (-90:270:0.15);
%\node at (1.6,-0.5) {$\frac{a}{b}$};
%\draw (1.6,-0.1) -- (1.6,0.1);
%
%\end{tikzpicture}
%%\caption{}
%\end{figure}
%
%
We distinguish three cases.
\begin{enumerate}
\item If $a/b  < p_{n-1}/q_{n-1} $, then obviously $R_{p_{n-1}/q_{n-1}}(x)<R_{a/b}(x)$.
%{\color{blue} (SB: I added the explanation:)}
\item Now assume $a/b > p_{n}''/q_{n}''$.
We note that, for $r/s\in\bQ$ and $t,$ $t'\in\bR$, 
\begin{equation}\label{Eq:pf6}|t-r/s|<|t'-r/s|  \text{ implies that }R_{r/s}(t)<R_{r/s}(t').\end{equation}
Then we have
\begin{equation}\label{Eq:pf7}
R_{{p_{n-1}}/{q_{n-1}}}(x) <R_{{p_{n-1}}/{q_{n-1}}}(p_{n}''/q_{n}'') =\text{rad}(C_{p_n''/q_n''}) \le R_{a/b}(p_{n}''/q_{n}'')  <R_{a/b}(x).
\end{equation}
The first and last inequalities in \eqref{Eq:pf7} follow from \eqref{Eq:pf6} and the fact $p_{n-1}/q_{n-1}<x<p''_{n}/q''_{n}<a/b$.
The equality in \eqref{Eq:pf7} holds since $C_{p''_n/q''_n}$ and $C_{p_{n-1}/q_{n-1}}$ are tangent to each other.
The second last inequality in \eqref{Eq:pf7} follows from \eqref{Eq:pf5}.
%\begin{equation}\begin{matrix*}[l]
%R_{{p_{n-1}}/{q_{n-1}}}(x) & <R_{{p_{n-1}}/{q_{n-1}}}(p_{n}''/q_{n}'') & (p_{n-1}/q_{n-1}<x<p''_{n}/q''_{n})\\
%					& =\text{rad}(C_{p_n''/q_n''}) &(\text{$C_{p''_n/q''_n}$ and $C_{p_{n-1}/q_{n-1}}$ are tangent to each other})\\
%					& \le R_{a/b}(p_{n}''/q_{n}'') &dd\\
%					& <R_{a/b}(x) &dd
%\end{matrix*}\end{equation}
%\les{Since \les{$p_{n-1}/q_{n-1}<x<p''_{n}/q''_{n}$}, 
%$$R_{{p_{n-1}}/{q_{n-1}}}(x)<R_{{p_{n-1}}/{q_{n-1}}}(p_{n}''/q_{n}'').$$}
%\les{Since $C_{p''_n/q''_n}$ and $C_{p_{n-1}/q_{n-1}}$ are tangent to each other, $R_{p_{n-1}/q_{n-1}}(p''_n/q''_n)=\text{rad}(C_{p_n''/q_n''})$.}
%\les{By \eqref{Eq:pf5}, $ \text{rad}(C_{p_{n}''/q_{n}''})\le R_{a/b}(p_{n}''/q_{n}'')$.} %(\les{the equality holds} when $C_{p_{n}''/q_{n}''}$ is tangent to $C_{a/b}$).
%\les{Thus, by \eqref{Lem:xbtw}, $R_{a/b}(p_{n}''/q_{n}'')<R_{a/b}(x)$.}
%\don{\sout{and $R_{{p_n}/{q_n}}(x)$ is proportional  to the distance of $x$ and $p_n/q_n$.}}
%Since \les{$x\in [p'_{n}/q'_{n},p''_{n}/q''_{n}]$}, $R_{{p_{n-1}}/{q_{n-1}}}(x)$ is between $R_{{p_{n-1}}/{q_{n-1}}}(p_{n}'/q_{n}')=1/2(q_{n}')^2$ and $R_{{p_{n-1}}/{q_{n-1}}}(p_{n}''/q_{n}'')=1/2(q_{n}'')^2$. 
%Since $x$ is between $p_{n+1}'/q_{n+1}'$ and $p_{n+1}''/q_{n+1}''$, the radius $R_{{p_n}/{q_n}}(x)$ is between $R_{{p_n}/{q_n}}(p_{n+1}'/q_{n+1}')=1/2(q_{n+1}')^2$ and $R_{{p_n}/{q_n}}(p_{n+1}''/q_{n+1}'')=1/2(q_{n+1}'')^2$. % since $R_{{p_n}/{q_n}}(x)$ is proportional to the distance of $x$ and $p_n/q_n$. 
%{\color{red} (please explain).}
%Thus, $R_{p_{n-1}/q_{n-1}}(x)<R_{a/b}(x)$.

\item Finally, let $a/b \in (p_{n-1}/q_{n-1},p_{n}'/q_{n}')$.
Denote by $C$ and $C'$ the horocycles based at $x$ tangent to $C_{a/b}$ and $C_{p_{n-1}/q_{n-1}}$, respectively (see Figure~\ref{Fg:bestapp_2}).
%(The thick dashed arc in Figure~\ref{Fg:bestapp_2} represents a part of $C_1$ and the thin dashed circle.)
Since the tangent point of $C$ and $C_{a/b}$ is an interior point of the area bounded by $C_{p_{n-1}/q_{n-1}}$, $C_{p_{n}'/q_{n}'}$ and the real line, we conclude that $C$ intersects $C_{p_{n-1}/q_{n-1}}$.
Thus, $C$ is larger than $C'$, i.e, $R_{p_{n-1}/q_{n-1}}(x)<R_{a/b}(x)$. \qedhere
\end{enumerate}
%Thus, for all the cases, $R_{a/b}(x)>R_{p_n/q_n}(x)$, i.e., $a/b$ is not a $1$-rational best approximation.
%If $R_w(x)\le R_{\frac{p_n}{q_n}}(x)$, then $w$ is between $p_{n+1}'/q_{n+1}'$ and $p_{n+1}''/q_{n+1}''$.
%We have $R_{\frac{p_n}{q_n}}(x)\le R_u(x)$.
%Thus $u$ is not a best approximating $1$-rational.
\end{proof}

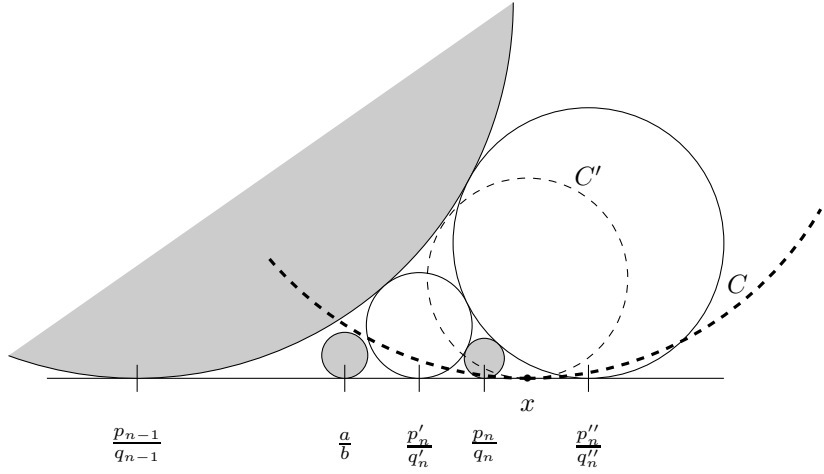
\begin{figure}[t]
\begin{tikzpicture}[scale=1.8]

\draw (0,0) -- (5,0);

%\draw (0.25*5,0) arc (-90:270:1.05);
%\node at (0.25*5,-0.3) {$v$};
%\draw (0.25*5,-0.1) -- (0.25*5,0.1);

%\filldraw[fill =white] (-16+50*1/3,0) arc (-90:-120:50*1/3*1/3*0.5);
\filldraw[fill = gray!40](-16+50*1/3+50*1/3*1/3*0.5,50*1/3*1/3*0.5) arc (0:-110:50*1/3*1/3*0.5);  \node at (-16+50*1/3,-0.5) {$\frac{p_{n-1}}{q_{n-1}}$}; \draw (-16+50*1/3,-0.1) -- (-16+50*1/3,0.1);
\hocirr{3}{8}{white} \node at (-16+50*3/8,-0.5) {$\frac {p_{n}'}{q_{n}'}$}; \draw (-16+50*3/8,-0.1) -- (-16+50*3/8,0.1);
\hocirr{2}{5}{white}  \node at (-16+50*2/5,-0.5) {$\frac {p_{n}''}{q_{n}''}$}; \draw (-16+50*2/5,-0.1) -- (-16+50*2/5,0.1); 
\hocirr{5}{13}{gray!40} \node at (-16+50*5/13,-0.5) {$\frac {p_{n}}{q_{n}}$}; \draw (-16+50*5/13,-0.1) -- (-16+50*5/13,0.1);

\draw[style = dashed] (3.55,0) arc (-90:270:0.74);
\draw[style = dashed, very thick] (3.55,0) arc (-90:-140:2.5);
\draw[style = dashed, very thick] (3.55,0) arc (-90:-30:2.5);
\node at (5.1,0.7) {$C$};
\node at (4, 1.5) {$C'$};

\node at (3.55,-0.2) {$x$};
\node at (3.55,0) {\tiny$\bullet$};

%\draw (3.8,-0.1) -- (3.8,0.1);

\filldraw[fill = gray!40] (2.2,0) arc (-90:270:0.17);
\node at (2.2,-0.5) {$\frac{a}{b}$};
\draw (2.2,-0.1) -- (2.2,0.1);

\end{tikzpicture}
\caption{A possible relative position of $x$, $a/b$ and the convergents. The dashed circles $C$ and $C'$ are horocycles based at $x$ tangent to $C_{a/b}$ and $C_{p_{n-1}/q_{n-1}}$.}
\label{Fg:bestapp_2}
\end{figure}

\section{Relation with the regular continued fraction}\label{Sec:RCF}
For simplicity, we denote a RCF as in \eqref{Eq:RCF} by $[d_0;d_1,d_2,\cdots,d_j,\cdots]$.
%$$[d_0;d_1,d_2,\cdots,d_j,\cdots] \quad\quad~\orr~\quad [d_0;d_1,d_2,\cdots,d_n].$$
For $1\le j \le d_{n}$ and for $n\ge 1$, the fractions
$$\frac{p^R_{n,j}}{q^R_{n,j}}=\frac{p^R_{n-2}+jp^R_{n-1}}{q^R_{n-2}+jq^R_{n-1}}$$ 
are called \emph{the intermediate convergents} %which are fractions $p/q$ satisfying the following Diophantine property
%$$|x-p/q|<|x-a/b| ~\forr~\any~a/b\not=p/q ~\st~ 0<q\le b,$$
%then $p/q$ is a intermediate convergents, but The converse does not hold
 (see \cite[Section 6]{Khi63} and \cite[p.36]{RoSz92}).
Kraaikamp and Lopes \cite{KrLo96} showed that the convergents of EICF are intermediate convergents of RCF.
In this section, we show that the OOCF principal convergents are also intermediate convergents of RCF.

The following lemma tells us how the piecewise inverses of OOCF act on RCF expansions.

\begin{lemma}\label{Lem:ChangeRCF}
Let $x=[0;d_1,d_2,\cdots]$.
Then, the RCF expansion of $f_{(a,\veps)}(x)$ is as follows:
$$f_{(a,\veps)}(x)=\left\{\begin{matrix*}[l]
[0;2,d_1,d_2,\cdots] &\If~ \veps=1,~a=1,\\
[0;1,(a-1),1,d_1,d_2,\cdots] &\If~\veps=1,~a\ge 2,\\
[0;(d_1+2),d_2,\cdots] &\If~\veps=-1,~a=2,\\
[0;1,(a-1),(d_1+1),d_2,\cdots] &\If~\veps=-1,~a\ge 3 .
\end{matrix*}\right.
$$
\end{lemma}

\begin{proof}

If $\veps=1$, then $$f_{(a,\veps)}(x)=1-[0;a,1,d_1,d_2,\cdots]=
\left\{\begin{matrix*}[l]
[0;2,d_1,d_2,\cdots] &\If~a=1,\\
[0;1,a-1,1,d_1,\cdots] &\If~a\ge 2.
\end{matrix*}\right.$$
If $\veps=-1$, then
\begin{equation*}\begin{split}
f_{(a,\veps)}(x) & =1-\cfrac{1}{a-[0;1,d_1,d_2,\cdots]}=1-\cfrac{1}{(a-1)+1-[0;1,d_1,d_2,\cdots]}\\
& =1-[0;a-1,d_1+1,d_2,\cdots]=
\left\{\begin{matrix*}[l]
[0;d_1+2,d_2,\cdots] &\If~a=2,\\
[0;1,a-1,1,d_1+1,d_2,\cdots] &\If~a\ge 3.
\end{matrix*}\right. \qedhere
\end{split}\end{equation*}
\end{proof}

Applying Lemma \ref{Lem:ChangeRCF}, we have the following theorem.
\begin{theorem}\label{Thm:intermediate}
The OOCF principal convergents of $x$ are intermediate convergents of $x$.
\end{theorem}

\begin{proof}
Let $x=[0;  d_1,d_2,\cdots]=[\![(a_1,\veps_1),(a_2,\veps_2),\cdots]\!]$.
%Then,
%$$f_{(a,\veps)}(x)=\begin{cases}
%[2,A_1,A_2,\cdots]~\If~ \veps=1,~a=1\\
%[1,(a-1),1,A_1,A_2,\cdots]~\If~\veps=1,~a\ge 2\\
%[A_1+2,A_2,\cdots]~\If~\veps=-1,~a=2\\
%[1,a-1,A_1+1,A_2,\cdots]~\If~\veps=-1,~a\ge 3 .
%\end{cases}
%$$
%The map $f_{(a,\veps)}$ is in (\ref{Eq:inverse}).
Note that $$x=f_{(a_1,\veps_1)}\circ f_{(a_2,\veps_2)}\circ \cdots \circ f_{(a_k,\veps_k)}([\![(a_{k+1},\veps_{k+1}),\cdots]\!])
\quad\text{and}\quad \frac{p_k}{q_k}=f_{(a_1,\veps_1)}\circ f_{(a_2,\veps_2)}\circ \cdots \circ f_{(a_k,\veps_k)} (1).$$
By Lemma~\ref{Lem:ChangeRCF}, $x$ and $p_k/q_k$ have the same prefix in their RCF expansions, except for the last partial quotient of $p_k/q_k$.
Thus, $p_k/q_k$ is an intermediate convergent of $x$.
\end{proof}

Next, we show that we can convert RCF expansions into OOCF expansions. Before we state the theorem, let us introduce the following notations:
$$\left\{\begin{matrix*}[l]
x=[0;d_1,\cdots,d_n,\tau]  &\text{ if } \  G^n(x)=\tau, \\
%$$\cfrac{1}{A_1+\cfrac{1}{\ddots+\cfrac{1}{A_n+\gamma}}}.$$
x=[\![(a_1,\veps_1),(a_2,\veps_2),\cdots,(a_n,\veps_n),\gamma]\!]  &\text{ if } \ \TO^n(x)=\gamma.
%$$
%x = 1 - \cfrac{1}{a_1 + \cfrac{\veps_1}{ 2 - \cfrac{1}{\ddots + \cfrac{\ddots}{ 2 - \cfrac{1}{a_n + \cfrac{\veps_n}{1+\gamma}}}}}},
%$$
\end{matrix*}\right.
$$

%\begin{theorem}
%If $A_1=1$, then
%%$B(k+1,-1)=\cup_{n\ge 2}[0;1,k-1,n]$ and $B(k,1)=[0;1,k-1,1]$
%$$
%[0;A_1,A_2,\alpha]=
%\begin{cases}
%[\![(A_2+1,1),F(\alpha)]\!], \text{ if }\alpha\in[\frac{1}{2},1),\\
%[\![(A_2+2,-1),F(\alpha)]\!], \text{ if }\alpha\in[0,\frac{1}{2}).\\ %\frac{1-\alpha}{\alpha}
%\end{cases}
%$$
%
%If $A_1\ge 2$, then
%\begingroup
%$$
%[0;A_1,\beta]=\left\{\vphantom{\begin{array}{@{}c@{}} \strut \\ \strut\end{array}}\right.\kern-\nulldelimiterspace
%\begin{array}{@{}l@{}}
%[\![\overbrace{(2,-1),\cdots,(2,-1)}^{\left(\frac{A_1-1}{2}\right)-\text{times}},\frac{1}{1+\beta}]\!] \text{ if } A_1 \text{ is odd, } \\
%{[\![}\underbrace{(2,-1),\cdots,(2,-1)}_{\left(\frac{A_1}{2}-1\right)-\text{times}},(1,1),{\beta}]\!] \text{ if } A_1 \text{ is even. }
%\end{array}
%%\begin{array}{@{}c@{}}
%%[\![\overbrace{(2,-1),\cdots,(2,-1)}^{\left(\frac{A_1-1}{2}\right)-\text{times}},\frac{1}{1+\beta}]\!] \text{ if } A_1 \text{ is odd, }\\
%%[\![\underbrace{(2,-1),\cdots,(2,-1)}_{\left(\frac{A_1}{2}-1\right)-\text{times}},(1,1),{\beta}]\!] \text{ if } A_1 \text{ is even. }
%%\end{array}
%$$
%\endgroup
%\end{theorem}

\begin{theorem}\label{Thm:convert}
We can convert RCF expansions into OOCF expansions by the following relations:
\begingroup
$$
x = [0; d_1, d_2,\tau]= %\left\{\vphantom{\begin{array}{@{}c@{}} \strut \\ \strut \\ \strut \end{array}}\right.\kern-\nulldelimiterspace
%\begin{array}{@{}l@{}l}
\left\{\begin{matrix*}[l]
[\![{(2,-1)}^{\frac{d_1-1}{2}},(d_2+1,1),F(\tau)]\!]  &\text{ if } \ d_1 \text{ is odd and }\tau\in[\frac{1}{2},1), \\
{[\![}{(2,-1)}^{\frac{d_1-1}{2}},(d_2+2,-1),F(\tau)]\!] & \text{ if } \  d_1 \text{ is odd and }
\tau\in[0,\frac{1}{2}), \\
{[\![}{(2,-1)}^{\frac{d_1}{2}-1},(1,1),G(x)]\!]  & \text{ if } \ d_1 \text{ is even. }
\end{matrix*}\right.
%\end{array}
$$
\endgroup

%\begingroup
%$$
%x = [0;A_1,A_2,\alpha]=\left\{\vphantom{\begin{array}{@{}c@{}} \strut \\ \strut \\ \strut \end{array}}\right.\kern-\nulldelimiterspace
%\begin{array}{@{}l@{}}
%[\![\overbrace{(2,-1),\cdots,(2,-1)}^{\left(\frac{A_1-1}{2}\right)-\text{times}},(A_2+1,1),F(\alpha)]\!] \text{ if } A_1 \text{ is odd and }\alpha\in[\frac{1}{2},1), \\
%{[\![}\overbrace{(2,-1),\cdots,(2,-1)}^{\left(\frac{A_1-1}{2}\right)-\text{times}},(A_2+2,-1),F(\alpha)]\!] \text{ if } A_1 \text{ is odd and }
%\alpha\in[0,\frac{1}{2}), \\
%{[\![}\underbrace{(2,-1),\cdots,(2,-1)}_{\left(\frac{A_1}{2}-1\right)-\text{times}},(1,1),G(x)]\!] \text{ if } A_1 \text{ is even. }
%\end{array}
%$$
%\endgroup
\end{theorem}

\begin{proof}
%Let $x=[0;1,A_2,\alpha]$.
Noting 
%\begin{equation*}\label{Eq:trans}
$\frac{1}{1+\tau}=1-\frac{1}{1+\frac{1}{\tau}}$,
%\end{equation*}
we have
$$[0;1,d_2,\tau]=\frac{1}{1+\frac{1}{d_2+\tau}}=1-\frac{1}{(d_2+1)+\tau}.$$
%Since $\tau=\frac{1}{[\tau^{-1}]+G(\tau)}$, 
If $\tau\in[\frac{1}{2},1)$, then $\tau=\frac{1}{1+G(\tau)}$.
If $\tau\in[0,\frac{1}{2})$, then $$\tau=1-\frac{1}{1+\frac{1}{[\tau^{-1}]-1+G(\tau)}}.$$
Thus $F(\tau)=\frac{1}{[\tau^{-1}]-1+G(\tau)}$ if $\tau\in[0,\frac{1}{2})$ and $F(\tau)=G(\tau)$ if $\tau\in(\frac12,1)$.
In the case of $d_1=1$, we have
\begin{equation}\label{Eq:5.3_1}
[0;1,d_2,\tau]=
\left\{\begin{matrix*}[l]
[\![(d_2+1,1),F(\tau)]\!],  &\text{ if }\tau\in[\frac{1}{2},1),\\
[\![(d_2+2,-1),F(\tau)]\!],  &\text{ if }\tau\in[0,\frac{1}{2}).\\ %\frac{1-\alpha}{\alpha}
\end{matrix*}\right.
\end{equation}
Similarly,  if $d_1=2$, then
\begin{equation}\label{Eq:5.3_2}x= [0;2,G(x)]=\frac{1}{2+G(x)}=1-\frac{1}{1+\frac{1}{1+G(x)}}=[\![(1,1),G(x)]\!].
\end{equation}
%and
%$$[0;3,\beta]=\frac{1}{3+\beta}=1-\frac{1}{2-\frac{1}{1+\frac{1}{1+\beta}}}=[\![(2,-1),\frac{1}{1+\beta}]\!].$$
If $d_1\ge 3$, i.e., $x\in(0,\frac13)$, then 
\begin{equation}\label{Eq:5.3}x=[0;d_1,G(x)]=\frac{1}{d_1+G(x)}=1-\frac{1}{1+\frac{1}{1+(d_1-2)+G(x)}}=1-\frac{1}{2-\frac{1}{1+\frac{1}{(d_1-2)+G(x)}}}=[\![(2,-1),\TO(x)]\!].
\end{equation}
Note that there are $n\in\bN$ and $r\in \{0,1\}$ such that $d_1-1 = 2n+r$. Since $\TO(x)=[0;d_1-2,G(x)]$,
by repeating the process in \eqref{Eq:5.3}, we have 
$$[0;d_1,G(x)] = [\![(2,-1)^{n}, \TO^{n}(x)]\!]\quad\andd\quad { \TO^{n}(x) = [0; r+1 ,G(x)]} .$$
Then we can complete the proof by \eqref{Eq:5.3_1} and \eqref{Eq:5.3_2}.
%$$[0;3,\beta]=\frac{1}{3+\alpha}=1-\frac{1}{2-\frac{1}{1+\frac{1}{1+\alpha}}}=[\![(2,-1),\frac{1}{1+\alpha}]\!]$$
%and
%$$[0;4,\beta]=\frac{1}{4+\alpha}=1-\frac{1}{2-\frac{1}{1+\frac{1}{2+\alpha}}}=[\![(2,-1),\frac{1}{2+\alpha}]\!]$$
\end{proof}

%\begin{corollary}
%Note that if $x\in[1/2,1)$, then $F \circ G^2(x)=\TO(x)$.
%If $x \in [0,1/2)$, then $G(x) = G \circ \TO^{[\frac{d_1-1}{2}]}(x)$.
%\end{corollary}

Since $p^R_n/q^R_n$ is a best approximation, if $p^R_n/q^R_n$ is a $1$-rational, then $p^R_n/q^R_n$ is a best $1$-rational approximation. Thus, by Theorem~\ref{Thm:main1}, $p^R_n/q^R_n$ is an OOCF convergent.
Now we check when an intermediate convergent is an OOCF convergent.
Keita  \cite{Kei17} proved the following propostion.
\begin{proposition*}[Keita, Proposition 1.2 in \cite{Kei17}]
We have
$$q^R_{n,0}=q^R_{n-2}<q^R_{n-1}\le q^R_{n,1}<\cdots <q^R_{n,d_n}=q^R_{n},$$
$$|q^R_{n,d_n}x-p^R_{n,d_n}|=|q^R_{n}x-p^R_{n}|<|q^R_{n-1}x-p^R_{n-1}|\le |q^R_{n,d_n-1}x-p^R_{n,d_n-1}|<\cdots <|q^R_{n,0}x-p^R_{n,0}|=|q^R_{n-2}x-p^R_{n-2}|.$$
\end{proposition*}

By the above proposition and Theorem~\ref{Thm:intermediate}, if $p^R_{n-1}/q^R_{n-1}$ is a $1$-rational, then $p^R_{n,j}/q^R_{n,j}$ is not an OOCF principal convergent for any $1\le j < d_n$.
If $p^R_{n-1}/q^R_{n-1}$ is an $\infty$-rational and $p^R_{n,j}/q^R_{n,j}$ is a $1$-rational, then $p^R_{n,j}/q^R_{n,j}$ is an OOCF principal convergent.

%\section{An invariant density of $T_\mathrm{OOCF}$}\label{Sec:InvMsr}
%Let $D_1$ be a map from $[0,1]$ to $\mathbb{R}^2$ which defined as
%$$
%t\mapsto \left(\frac{1-t^2}{1+t^2},\frac{2t}{1+t^2}\right).
%$$
%This is a projection from a segment of $y$-axis whose end points are $(0,0)$ and $(0,1)$ to a unit circle with the center of projection $(-1,0)$.
%Let $D_2$ be a map from $[0,1]$ to $\mathbb{R}^2$ which defined as
%$$
%s \mapsto \left(\frac{2s}{1+t^2},\frac{1-s^2}{1+s^2}\right).
%$$
%This is a projection from a segment of $x$-axis whose end points are $(0,0)$ and $(1,0)$ to a unit circle with the center of projection $(0,-1)$.

%Then
%$$
%t=D_1^{-1}D_2(s)=\frac{1-s}{1+s}\quad \text{ and } \quad
%\left|\frac{\text{d}t}{\text{d}s}\right|=\frac{2}{(1+s)^2}.
%$$
%An invariant density of $T_\mathrm{EICF}$ is
%$$
%\left(\frac{1}{s+1}-\frac{1}{s-1}\right)\text{d}s.
%$$
%Since $T_\mathrm{OOCF}$ is a conjugation of $T_\mathrm{EICF}$ by $D_1^{-1}D_2$ and
%$$
%\frac{2|\text{d}s|}{1-s^2}=\frac{2|\text{d}t|}{1-s^2}\cdot\frac{(1+s)^2}{2}=\frac{|\text{d}t|}{t},
%$$
%an invariant densidy of $T_\mathrm{OOCF}$ is
%$
%\frac{\text{d}t}{t}.
%$

%{\color{red} Should clean the references}

\section*{Acknowledgements} 
We thank the referee for her/his valuable remarks and suggestions which significantly improve the presentation of the paper.  
D. K. and S. L. were supported by the National Research Foundation of Korea (NRF-2018R1A2B6001624).
S. L. also acknowledges the support of the Centro di Ricerca Matematica Ennio de Giorgi and of UniCredit Bank R\&D group for financial support through the Dynamics and Information Theory Institute at the Scuola Normale Superiore.


\begin{thebibliography}{60}
\bibitem{Alp05}{
   Alperin, R.C.: 
   The modular tree of Pythagoras.
   Amer. Math. Monthly.
   {\bf 112}(9), 807--816
   (2005)
%   issn={0002-9890},
%   review={\MR{2179860}},
%   doi={10.2307/30037602},
}

\bibitem{Bar63}{
   {Barning, F.J.M.}:
   {On Pythagorean and quasi-Pythagorean triangles and a generation process with the help of unimodular matrices}.
   {Math. Centrum Amsterdam Afd. Zuivere Wisk.}
   {\bf 1963}(ZW-011),
   37 (1963) 
%   (in Dutch).
%   review={\MR{0190077}},
}

\bibitem{Ber34}{
  Berggren, B.:
  Pytagoreiska triangular. 
  {Tidskr. Element\"ar Mat. Fys. Kemi}. 
%  {\it Tidskrift f\"or Element\"ar Matematik, Fysik och Kemi},
  {\bf 17}, 129--139 (1934) 
%10
}

\bibitem{BoMe18}{
  {Boca, F. P.}, {Merriman, C.}:
  {Coding of geodesics on some modular surfaces and applications to odd and even continued fractions}.
  {Indag. Math.}
  {\bf 29}(5), {1214--1234} {(2018)} 
%   issn={0019-3577},
%   review={\MR{3853422}},
%   doi={10.1016/j.indag.2018.05.004},
}

\bibitem{CK1}{
  {Cha, B.},
  {Kim, D.H.}:
  {Number theoretical properties of Romik's dynamical system}.
  {Bull. Korean Math. Soc.}
    {\bf 57}(1), {251--274} (2020)
%  doi={https://doi.org/10.4134/BKMS.b190163}, 
}

\bibitem{CK2}{
  {Cha, B.},
  {Kim, D.H.}:
  { Intrinsic Diophantine approximation of a unit circle and its Lagrange spectrum} (2019)
  (available at \href{https://arxiv.org/abs/1903.02882}{\textsf{arXiv:1903.02882}})
}

\bibitem{CNT}{
  {Cha, B.},
  {Nguyen, E.},
  {Tauber, B.}:
  {Quadratic forms and their Berggren trees}.
  {J. Number Theory}.
    {\bf 185}, {218--256} (2018) 
%  doi={https://doi.org/10.1016/j.jnt.2017.09.003},
}

\bibitem{Con}{
	{Conrad, K.}:
	{Pythagorean descent}. (2007)\\
	(available at \href{http://www.math.uconn.edu/~kconrad/blurbs/linmultialg/descentPythag.pdf}{\textsf{http://www.math.uconn.edu/{\~{}}kconrad/blurbs/linmultialg/descentPythag.pdf}})
}

\bibitem{ito1989algorithms}{
   {Ito, S.}: 
   {Algorithms with mediant convergence and their metrical theory}. 
   {Osaka J. Math}.
    {\bf 26}(3), {557--578} (1989)
}

\bibitem{JaLi88}{
   {Jager, H.},
   {Liardet, P.}:
   {Distributions arithm\'{e}tiques des d\'{e}nominateurs de convergents de fractions continues}.
  {Indag. Math.},
  {\bf 91}(2), {181--197} (1988) 
%   issn={0019-3577},
%   review={\MR{952514}},
}

%\bibitem{GLMWY}{
%   {Graham, Ronald L.},
%   {Lagarias, Jeffrey C.},
%   {Mallows, Colin L.},
%   {Wilks, Allan R.},
%   {Yan, Catherine H.},
%   {Apollonian circle packings: number theory},
%   {J. Number Theory},
%    {100},
%    {2003},
%    {1},
%    {1--45},
%   issn={0022-314X},
%   review={\MR{1971245}},
%   doi={10.1016/S0022-314X(03)00015-5},
%}

\bibitem{Kei17}{
   {Keita, A.}:
   {Continued fractions and parametric geometry of numbers}.
  {J. Th\'eor. Nombres Bordeaux}.
  {\bf 29}(1), {129--135}  (2017)
%   issn={1246-7405},
%   review={\MR{3614519}},
}

\bibitem{Khi63}{
  {Khinchin A.Y.}:
  {Continued fractions, translated from the 3rd (1961) Russian ed., reprint of the 1964 translation}. 
  Dover Publications, (1997)
}

\bibitem{Kra91}{
   {Kraaikamp, C.}:
   {A new class of continued fraction expansions}.
  {Acta Arith.}
  {\bf 57}(1), {1--39} (1991)
%   issn={0065-1036},
%   review={\MR{1093246}},
%   doi={10.4064/aa-57-1-1-39},
}

\bibitem{KrLo96}{
   {Kraaikamp, C.},
   {Lopes, A.}:
   {The theta group and the continued fraction expansion with even partial quotients},
  {Geom. Dedicata}.
  {\bf 59}(3), {293--333} {(1996)}
%   issn={0046-5755},
%   review={\MR{1371228}},
%   doi={10.1007/BF00181695},
}

\bibitem{Moe82}{
   {Moeckel, R.}:
   {Geodesics on modular surfaces and continued fractions}.
 {Ergod. Th. \& Dynam. Sys.},
  {\bf 2}(1), {69--83} (1982)
%   issn={0143-3857},
%   review={\MR{684245}},
%   doi={10.1017/s0143385700009585},
}

\bibitem{Pan20}{
 {Panti, G.}:
 {Billiards on pythagorean triples and their Minkowski functions}.
 Discrete Contin. Dyn. Syst.
 {\bf 40}(7), 4341--4378 (2020)
}


\bibitem{RoSz92}{  
  {Rockett A.M., Sz\"usz P.}:
  {Continued fractions}.
  World Scientific Publishing Co.
 {(1992})

}

\bibitem{Rom08}{
   {Romik, D.}:
   {The dynamics of Pythagorean triples}.
   {Trans. Amer. Math. Soc.}
 {\bf 360}(11), {6045--6064} (2008)
%   issn={0002-9947},
%   review={\MR{2425702 (2009i:37101)}},
%   doi={10.1090/S0002-9947-08-04467-X},
}

\bibitem{Sch82}{
  {Schweiger, F.}:
  {Continued fractions with odd and even partial quotients}.
  {Arbeitsber. Math. Inst. Univ. Salzburg}.
  {\bf 4}, {59--70} (1982)
}

\bibitem{Sch84}{
  {Schweiger, F.}:
  {On the approximation by continued fractions with odd and even partial quotients}.
  {Arbeitsber. Math. Inst. Univ. Salzburg}.
  {\bf 1}(2), {105--114} (1984)
}

\bibitem{schweiger1995ergodic}{
  Schweiger, F.:
  {Ergodic theory of fibred systems and metric number theory}. 
  Oxford University Press, 
  {(1995)}
}

\bibitem{series1985modular}{
  {Series, C.}:
  {The modular surface and continued fractions}.
  {J. London Math. Soc.}
  {{\bf s2-31}(1)}, {69--80} {(1985)}  
}

\bibitem{ShWa14}{
	{Short, I.},
	{Walker, M.}:
	{Even-integer continued fractions and the Farey tree}.
  In: {Symmetries in Graphs, Maps, and Polytopes Workshop}
  {(Springer Proc. in Math. Stat. 159)},
  Springer, (2016)
  {pp. 287--300}
}
\end{thebibliography}
\end{document}